\documentclass{amsart}
\usepackage{setspace}
\usepackage[english]{babel}
\usepackage[utf8x]{inputenc}
\usepackage[T1]{fontenc}
\usepackage{todonotes}
\usepackage{tikz}

\usepackage[a4paper,top=3cm,bottom=3cm,left=3cm,right=3cm,marginparwidth=1.75cm]{geometry}

\usepackage{graphicx}
\usepackage{amsmath,amsthm,amsfonts}
\usepackage{amssymb}
\usepackage{mathtools}
\usepackage{graphicx}
\usepackage{stmaryrd}
\usepackage{scalerel}
\usepackage{mathrsfs}
\usepackage{relsize}
\usepackage{enumerate}   
\usepackage{bm,dsfont} 
\usepackage{IEEEtrantools}

\usepackage{commath}

\usepackage{varioref}
\usepackage{hyperref}
\usepackage[nameinlink, capitalize, noabbrev]{cleveref}

\theoremstyle{plain} %% needs `amsmath' package
\newtheorem{thm}{Theorem}[section]
\newtheorem{lemma}[thm]{Lemma}
\newtheorem{prop}[thm]{Proposition}

\newtheorem{exmp}[thm]{Example}
\newtheorem{conv}[thm]{Convention}

\theoremstyle{definition} %% needs `amsmath' package
\newtheorem{defn}[thm]{Definition}
\theoremstyle{remark} %% needs `amsmath' package
\newtheorem{rmk}[thm]{Remark}
\numberwithin{equation}{section}

\DeclareMathOperator{\ad}{ad}     %% adjoint derivation
\DeclareMathOperator{\Dom}{Dom}   %% domain of an operator
\DeclareMathOperator{\linspan}{span} %% linear span
\DeclareMathOperator{\tr}{tr}     %% matrix trace
\DeclareMathOperator{\blangle}{{}_{\bullet}\langle} %%Bullet langle
\DeclareMathOperator{\brangle}{\rangle_{\bullet}} %%Bullet rangle
\newcommand{\A}{\mathcal{A}}		%% Smooth subalgebra
\newcommand{\B}{\mathcal{B}}		%% Smooth subalgebra
\newcommand{\R}{\mathbb{R}}			%% Real numbers
\newcommand{\C}{\mathbb{C}}			%% Complex numbers
\newcommand{\Q}{\mathbb{Q}}			%% Rational numbers
\newcommand{\Z}{\mathbb{Z}}			%% Integers
\newcommand{\N}{\mathbb{N}}			%% Natural numbers
\newcommand{\E}{\mathcal{E}}        %% Hilbert module
\newcommand{\GS}{\mathcal{G}}		    %% Gabor system
\newcommand{\La}{\Lambda}    			%% short for \Lambda
\newcommand{\la}{\lambda}     			%% short for \lambda
\newcommand{\Lao}{\Lambda^{\circ}}      %% short for \Lambda circ
\newcommand{\lao}{\lambda^{\circ}}      %% short for \lambda circ
\newcommand{\Sub}{\Delta}     			%% Closed subspace
\newcommand{\ZZ}{\Z}
\newcommand{\RR}{\mathbb{R}}

\def\hs#1#2{\left\langle #1,#2\right\rangle} % Inner product
\def\lhs#1#2{{_\bullet\!\!}\left\langle #1,#2\right\rangle} % Algebra valued left inner product
\def\rhs#1#2{\left\langle #1,#2\right\rangle\!\!{_\bullet}}	% Algebra valued right inner product
\def\modft{\Theta}

\def\tfp#1{#1 \times \widehat{#1}} % Time-frequency plane associated to the group #1
\def\heis#1#2{E_{#1,#2}}
\def\Cdif{C_{\mathrm{dif}}}
\def\Csm{C_{\mathrm{sm}}}
\def\Cgr{C_{\mathrm{gr}}}

\newcommand{\Qp}{\mathbb{Q}_{p}} % the p-adic numbers
\newcommand{\Zp}{\mathbb{Z}_{p}} % the p-adic integers
\newcommand{\SO}{\textnormal{\textbf{S}}_{0}} % the Feichtinger algebra
\usepackage[foot]{amsaddr}

\author{Are Austad}
\author{Franz Luef}
\address{Norwegian University of Science and Technology, Department of Mathematical Sciences, Trondheim, Norway.}
\email{are.austad@ntnu.no, franz.luef@ntnu.no}

\title{Modulation Spaces as a Smooth Structure in Noncommutative Geometry}

\begin{document}
\maketitle

\begin{abstract}
We demonstrate how to construct spectral triples for twisted group $C^*$-algebras of lattices in phase space of a second countable locally compact abelian group by using a class of weights appearing in time-frequency analysis. This yields a way of constructing quantum $C^k$-structures on Heisenberg modules, and we show how to obtain such structures by using Gabor analysis and weighted analogues of Feichtinger's algebra. 
We treat %show that 
the standard spectral triple for noncommutative 2-tori %is
as a special case, and as another example we define a spectral triple on noncommutative solenoids and a quantum $C^k$-structure on the associated Heisenberg modules. 
\end{abstract}

%\tableofcontents

\section{Introduction}

The interplay between Gabor analysis and noncommutative geometry \cite{Co94} has been explored earlier and has recently attracted some interest, see for example \cite{AuEn19,AuJaLu19,En19,EnJaLu19,JaLu18,JaLu18-1,Kr16,Lu09,Lu18}. Indeed, problems in Gabor analysis can often effectively be rephrased as operator algebraic questions. Moreover, Gabor analysis provides a way to generate projective modules over noncommutative tori \cite{Lu09}. Hence, results in Gabor analysis supply interesting examples of structures studied in operator algebra theory and noncommutative geometry. The main part of this paper focuses on the latter aspect.

We are interested in smooth structures on the level of projective modules over $C^*$-algebras, which we call {\it quantum $C^k$-structures}, or \emph{$QC^k$-structures}. These are based on $QC^k$-structures on spectral triples for these $C^*$-algebras. Our focus is on Heisenberg modules over twisted group $C^*$-algebras of lattices in $G\times\widehat{G}$ for a second countable locally compact abelian group $G$. 

In terms of Gabor analysis the notion of $QC^k$-modules over noncommutative tori translates into better time-frequency localization of the window function generating the frame. It is common to refer to a Gabor frame generated by a Gaussian as better than one generated by e.g., a triangle function. Our results turn this observation into a rigorous statement and weighted analogues of Feichtinger's algebra appear naturally in this context. 

We discuss in detail the noncommutative 2-torus and noncommutative solenoids, introduced in \cite{LaPa13,LaPa15}. For the noncommutative 2-torus we show that our approach yields an equal $QC^k$-structure as if using the standard spectral triple, and for the noncommutative solenoid our construction provides a definition of smoothness which so far has not appeared in the literature. Note that the smooth structures introduced by Connes for noncommutative tori are also smooth in our sense but his approach does not allow one to identify structures with a fixed regularity like $QC^k$-structures.

In Section 2 we review relevant material on Hilbert $C^*$-modules and standard module frames with a focus on equivalence bimodules describing Morita equivalent $C^*$-algebras. In the next section we introduce $QC^k$-structures on $C^*$-algebras coming from spectral triples and %for finitely generated projective 
the corresponding structures on Hilbert $C^*$-modules. % over a $C^*$-algebra. 
%In case we have $QC^k$-structures for all $k\in\NN$ we call the structures \emph{smooth}. 
Section 4 contains the basics on Gabor frames for lattices in $G\times\widehat{G}$ for a second countable locally compact abelian group $G$, and we define Feichtinger's algebra $M^1 (G)$, the prime example of a modulation space, and weighted variants $M^1_v(G)$ for a natural class of weights on $G\times\widehat{G}$. Section 5 contains the main results of the paper: (i) The construction of $QC^k$-structures on twisted group $C^*$-algebras of lattices in  $G\times\widehat{G}$, and (ii) a description of $QC^k$-structures on Heisenberg modules, and that these are just weighted Feichtinger algebras. At the end of the section we treat the noncommutative 2-torus and the noncommutative solenoid in detail. We provide examples of projective modules that are $QC^k$ but not $QC^{k+1}$, and some that are %not 
smooth.
\section{Preliminaries}
\label{sec:preliminaries}
This section is dedicated to reminding the reader about module frames and Morita equivalence. We assume basic knowledge about $C^*$-algebras, Banach $*$-algebras, and their modules.

In the sequel we will let the $C^*$-algebra valued inner product on a left Hilbert $C^*$-module be denoted by $\lhs{\cdot}{\cdot}$, and likewise the $C^*$-algebra valued inner product on a right Hilbert $C^*$-module will be denoted by $\rhs{\cdot}{\cdot}$.

For the purposes of this paper it will suffice to look at finite module frames.
\begin{defn}\label{def:finite-frame}
	Let $A$ be a $C^*$-algebra, let $E$ be a left Hilbert $A$-module, and let $( g_i)_{i=1}^l$ be a sequence in $E$. 
	We say $( g_i)_{i=1}^l$ is a \textit{module frame for $E$} if there exist constants $C,D > 0$ such that
	\begin{equation}
	\label{Eq:Modular-frame-ineq}
	C \lhs{f}{f} \leq \sum_{i=1}^l \lhs{f}{g_i}\lhs{g_i}{f} \leq D \lhs{f}{f}
	\end{equation}
	as elements of $A$ for all $f \in E$. If $C=D=1$ we say $( g_i)_{i=1}^l$ is a \emph{Parseval module frame for $E$}.
\end{defn}
For a left Hilbert $A$-module $E$ we associate to any finite sequence $(g_i)_{i=1}^l \subset E$ the $A$-adjointable operator
%We associate to any finite sequence $(g_i)_{i=1}^l \subset E$, for $E$ a left Hilbert $A$-module, the $A$-adjointable operator
\begin{equation}
\begin{split}
    \modft_{(g_i)} : E &\to E \\
    f &\mapsto \sum_{i=1}^l \lhs{f}{g_i}g_i.
\end{split}
\end{equation}
This operator is called the \emph{frame operator of $(g_i)_{i=1}^l$}. Note that the frame operator is a positive operator on $E$ as $\lhs{\modft_{(g_i)}f}{f}\geq 0 $ for all $f \in E$. The following is a special case of \cite[Theorem 1.2]{ArBa17}
\begin{prop}
    Let $(g_i)_{i=1}^l$ be a sequence in a left Hilbert $A$-module $E$.
	%Let $A$ be a $C^*$-algebra, let $E$ be a left Hilbert $C^*$-module, and let $(g_i)_{i=1}^l$ be a sequence in $E$. 
	Then $(g_i)_{i=1}^l$ is a module frame for $E$ if and only if $\modft_{(g_i)}:E \to E$ is invertible. 
\end{prop}
\begin{defn}\label{defn:dual-and-tight-frame}
	Let $E$ be a left Hilbert $A$-module and let $( g_i)_{i=1}^l \subset E$ be a frame. Denote by $\modft$ the frame operator of $( g_i)_{i=1}^l $. We say $(\modft^{-1} g_i)_{i=1}^l$ is the \emph{canonical dual frame of $( g_i)_{i=1}^l$}, and we say $(\modft^{-1/2} g_i)_{i=1}^l$ is the \emph{canonical Parseval frame associated to $(g_i)_{i=1}^l$}.
\end{defn}
Given a frame $(g_i)_{i=1}^l$ for a left Hilbert $A$-module $E$, with frame operator $\modft$, we see that the canonical dual frame $(\modft^{-1} g_i)_{i=1}^l$ has the property that
\begin{equation}
	f = \sum_{i=1}^{l} \lhs{f}{g_i} \modft^{-1} g_i = \sum_{i=1}^{l} \lhs{f}{\modft^{-1} g_i} g_i
\end{equation}
for all $f \in E$. Indeed, this follows by writing out $f = \modft^{-1} \modft f = \modft \modft^{-1} f$. Any sequence $(h_i)_{i=1}^l$ such that
\begin{equation}
f = \sum_{i=1}^l \lhs{f}{g_i} h_i
\end{equation}
will be called a \emph{dual sequence of} $( g_i)_{i=1}^l$. %It follows by \cite[Proposition 6.3]{FrLa02} that the equality 
%\begin{equation*}
%	f = \sum_{i=1}^l \lhs{f}{g_i} h_i
%\end{equation*} 
%holds too.
Likewise, if we write out  $f = \modft^{-1/2} \modft \modft^{-1/2} f$, we get that the canonical Parseval frame associated to $(g_i)_{i=1}^l$ has the property
\begin{equation}
	f = \sum_{i=1}^{l} \lhs{f}{\modft^{-1/2} g_i} \modft^{-1/2} g_i 
\end{equation}
for all $f \in E$, and is a Parseval module frame as in \cref{def:finite-frame}. 	

The following result follows from \cite[Proposition 3.9]{Jing06}. There it is assumed the $C^*$-algebra is unital, but we include a weakened version of the result so that it is clear that this assumption can be dropped.
\begin{prop}\label{prop:dual-seq-implies-frame}
    Let $E$ be a Hilbert $A$-module and let $(g_i)_{i=1}^l$ and $(h_i)_{i=1}^l$ be sequences in $E$. If 
    \begin{equation*}
        f = \sum_{i=1}^l \lhs{f}{g_i}h_i
    \end{equation*}
    for all $f \in E$, then $(h_i)_{i=1}^l$ is a frame for $E$.
\end{prop}

The modules of interest in this paper will be Morita equivalence bimodules. For a reference on Morita equivalence of $C^*$-algebras we refer the reader to \cite{RaWi98}. 
\begin{defn}
	\label{Definition: Equivalence bimodule}
	Let $A$ and $B$ be $C^*$-algebras. A \emph{Morita equivalence bimodule between} $A$ and $B$, or an \emph{$A$-$B$-equivalence bimodule}, is a Hilbert $A$-$B$-bimodule $E$ satisfying the following conditions:
	\begin{enumerate}
		\item $\overline{\blangle E,E\rangle}=A$  and $\overline{\langle E , E \brangle} = B$, where $\blangle E,E\rangle = \linspan_\C \{ \blangle f,g \rangle \mid f,g\in E\}$ and likewise for $\langle E,E\brangle$.
		\item For all $f,g \in E$, $a \in A$ and $b \in B$,
		\begin{equation*}
		\text{$\langle a f, g\brangle = \langle f, a^* g \brangle$  and $\blangle fb, g\rangle = \blangle f,gb^*\rangle$.}
		\end{equation*}
		\item For all $f,g,h \in E$,
		\begin{equation*}
		\blangle f,g \rangle h = f\langle g,h \brangle.
		\end{equation*}
	\end{enumerate}
	
	Now let $\A \subset A$ and $\B \subset B$ be dense Banach $*$-subalgebras such that the enveloping $C^*$-algebra of $\A$ is $A$, and the enveloping $C^*$-algebra of $\B$ is $B$. Suppose further that there is a dense $\A$-$\B$-inner product submodule $\E \subset E$ such that the conditions above hold with $\A,\B,\E$ instead of $A,B,E$. In that case we say $\E$ is an \textit{$\A$-$\B$-pre-equivalence bimodule}.
\end{defn}
We have the following important result. A proof can be found in \cite[Proposition 2.11]{AuJaLu19}, but the result is well-known and dates back further.
\begin{prop}
	\label{prop:fin-gen-proj-iff-unital}
	Let $E$ be an $A$-$B$-equivalence bimodule. Then $E$ is a finitely generated projective $A$-module if and only if $B$ is unital.
\end{prop}
Module frames in Morita equivalence bimodules were extensively studied in \cite{AuJaLu19}. We summarize the results we will need.
\begin{prop}\label{prop:results-from-aujalu}
	Let $E$ be an $A$-$B$-equivalence bimodule where $B$ is unital, with an $\A$-$\B$-pre-equivalence bimodule $\E \subset E$. Moreover, let $(g_i)_{i=1}^l$ be a sequence in $E$ and let $\modft$ denote the frame operator of $(g_i)_{i=1}^l$. Then the following hold:
	\begin{enumerate}
		\item [i)] $(g_i)_{i=1}^l$ is a module frame for $E$ as a Hilbert $A$-module if and only if $\sum_{i=1}^l \rhs{g_i}{g_i}$ is invertible in $B$.
		\item [ii)] If $(g_i)_{i=1}^l$ is a module frame for $E$ as an $A$-module, then the canonical dual is given by $(h_i)_{i=1}^l$, where
		\begin{equation}
			h_j = \modft^{-1} g_j = g_j \bigg(\sum_{i=1}^{l} \rhs{g_i}{g_i} \bigg)^{-1}
		\end{equation}
		for all $j=1, \ldots, l$, and the canonical tight frame associated to $(g_i)_{i=1}^l$ is given by $(h'_i)_{i=1}^l$, where
		\begin{equation}
			h'_j = \modft^{-1/2} g_j = g_j \bigg(\sum_{i=1}^{l} \rhs{g_i}{g_i} \bigg)^{-1/2} 
		\end{equation}
		for all $j =1, \ldots ,l$.
		\item [iii)]  Suppose $\B \subset B$ is spectral invariant with the same unit, and that $(g_i)_{i=1}^l$ is a module frame for $E$ as an $A$-module with $g_i \in \E$ for all $i=1, \ldots, l$. Then $\modft^{-1} g_i \in \E$ and $\modft^{-1/2} g_i \in \E$ for all $i=1,\ldots ,l$. 
	\end{enumerate}
\end{prop}
\begin{proof}
	Statement i) is immediate by \cite[Remark 3.27]{AuJaLu19} and \cite[Proposition 3.20]{AuJaLu19}. Since the action of $\modft$ is implemented by right multiplication by $\sum_{i=1}^{l} \rhs{g_i}{g_i} $ by \cite{AuJaLu19}, it follows that  $\modft^{-1}$ is implemented by right multiplication by $(\sum_{i=1}^{l} \rhs{g_i}{g_i} )^{-1}$, and the action of $\modft^{-1/2}$ is implemented by right multiplication by $(\sum_{i=1}^{l} \rhs{g_i}{g_i} )^{-1/2}$. Hence statement ii) follows as well. In statement iii) the fact that $\modft^{-1} g_i \in \E$ for all $i=1, \ldots ,l$ is immediate by \cite[Remark 3.27]{AuJaLu19} and \cite[Proposition 3.21]{AuJaLu19}. But if $(\sum_{i=1}^{l} \rhs{g_i}{g_i})^{-1} \in \B$ by spectral invariance, so is $(\sum_{i=1}^{l} \rhs{g_i}{g_i} )^{-1/2}$. So it follows that $\modft^{-1/2} g_i \in \E$ for all $i=1,\ldots ,l$ also.
\end{proof}

\section{Smoothness in noncommutative geometry}
\label{Section: Smoothness in NCG}
We dedicate this section to presenting a notion of smoothness used in noncommutative geometry. Given a $C^*$-algebra $A$ we fix a spectral triple $(\A, \mathcal{H}, D)$ for $A$, where $\A\subset A$ is a dense $*$-subalgebra, $\mathcal{H}$ is a Hilbert space and $D: \mathcal{H} \to \mathcal{H}$ is a densely defined selfadjoint operator. 

The concept of regular spectral triples was introduced by Connes in \cite{Co95}, but we adopt the terminology of quantum $C^k$ spectral triples used in \cite{caphre11}.
\begin{defn}
	\label{def:QCk-structure-alg}
	Let $A$ be a $C^*$-algebra and let $(\A, \mathcal{H}, D)$ be a spectral triple for $A$. We say $(\A, \mathcal{H}, D)$ is \textit{quantum $C^n$}, or $QC^n$, $n \in \N$, if for all $a \in \A$ both $a$ and $[D,a]$ are in the domain of $\ad^n(\vert D \vert)$. Here $\ad^j (|D|)(a)$ is the $j$ times iterated commutator $[|D|,[|D|,\ldots,[|D|,a]\ldots]]$, $j\in \N$. If  $(\A, \mathcal{H}, D)$ is $QC^n$ for all $n \in \N$, we say it is $QC^{\infty}$. 
\end{defn}
With this definition we obtain a notion of smoothness on the $C^*$-algebra $A$, namely, for any $n \in \N$ we set
%\begin{equation}
%	Q\A_1 := \{a \in A \mid \textrm{$a \cdot \Dom D \subset \Dom D$ and $[D,a]$ extends to an element of $\B (\mathcal{H})$}\}
%\end{equation}
%and
\begin{equation}
	QA_n := \{a \in \A \mid \text{both $a$ and $[D,a]$ are in $\Dom (\ad^n(\vert D\vert))$}   \}.
\end{equation}
With a $QC^n$-structure on a $C^*$-algebra $A$ we can, for any Hilbert $A$-module $E$, specify natural $QC^n$-submodules.% for all $n \geq 1$. 
\begin{defn}
	\label{def:QCk-structure-mods}
	Let $A$ be a $C^*$-algebra equipped with a $QC^n$ spectral triple for some $n \geq 1$, and let $E$ be a left Hilbert $A$-module. Suppose there exists a uniformly norm bounded approximate unit $(e_m)_{m=1}^{\infty}$ for $E$, with
	\begin{equation}\label{eq:approx-id-form}
		e_m = \sum_{i=1}^m \Theta_{g_i,g_i}.
	\end{equation}
    Here $\Theta_{g,h}$ is the rank one module operator $\modft_{g,h}f = \lhs{f}{g} h$. We say $(E, (e_m)_{m=1}^{\infty})$ is a \textit{$QC^n$-$A$-module} if 
	$\lhs{g_i}{g_j} \in QA_n$ for all $i,j \in \{1,\ldots,m\}$ and all $m \in \N$. If $(E, (e_m)_{m=1}^{\infty})$ is a $QC^k$-$A$-module for all $k\in \N$, we say $(E, (e_m)_{m=1}^{\infty})$ is a \textit{$QC^{\infty}$-$A$-module}.
\end{defn}
The above definition is inspired by the definition of $C^k$-modules in \cite{Mes14}. %However, we do not take the operator algebra perspective as done in said article, but rather require the approximate identity to be uniformly norm bounded in $C^*$-norm. 

\section{Gabor analysis on LCA groups and weighted Feichtinger algebras}
\label{sec:Gab-An-LCA-weight-Feich-algs}
Before discussing the mathematical objects of interest we recall some central concepts from Gabor analysis on locally compact abelian (LCA) groups. 

Througout this section we fix a second countable LCA group $G$, and we will let $\La$ be a lattice in $\tfp{G}$, that is, $\La$ is a cocompact and discrete subgroup in $\tfp{G}$. Here $\widehat{G}$ is the dual group of $G$. The group $\tfp{G}$ is sometimes called the \emph{time-frequency plane of $G$} or the \emph{phase space of $G$}. We will have to restrict to lattices $\La$ as we wish to make use of the localization procedure developed in \cite{AuEn19} in a particular way. Namely, we need to be able to localize both the $C^*$-algebra $C^* (\La,c)$ and a Heisenberg module, both defined in \cref{Section:twisted-group-algebras-and-Ck}. 

Given $G$ and $\La$ we will need to make some choices regarding the Haar measures and how they relate to one another. The convention we will use is summed up in the following. 
\begin{conv}
	\label{conv:measure-choices}
	Given an LCA group $G$ we fix a Haar measure $\mu_G$ on $G$ and normalize the Haar measure $\mu_{\widehat{G}}$ on $\widehat{G}$ such that the Plancherel theorem holds. The lattice $\La$ will be equipped with the counting measure. On $(\tfp{G})/\La$ we put the Haar measure such that Weil's formula holds, that is, such that for all $f \in L^1 (G\times \widehat{G})$ we have
\begin{equation*}
\int_{G\times \widehat{G}} f(\xi) d\mu_{G\times \widehat{G}}(\xi) = \int_{(G\times \widehat{G})/\La} \int_{\La} f(\xi + \la ) d\mu_{\La}(\la) d\mu_{(G\times \widehat{G})/\La}(\Dot{\xi}), \quad \text{$\Dot{\xi} = \xi + \La$}.
\end{equation*}
\end{conv}
\begin{defn}
    The \emph{size of $\La$}, denoted $s(\La)$, is defined as
    \begin{equation*}
s(\La) = \int_{(G \times \widehat{G})/\La} 1 d\mu_{(G \times \widehat{G})/\La}.
\end{equation*}
\end{defn}
\begin{rmk}
    When $\La$ is a lattice it is in particular cocompact. Hence $(\tfp{G})/\La$ is compact, which implies $s(\La) < \infty$.
\end{rmk}
For any point $\xi = (x,\omega)\in \tfp{G}$ we define the \emph{time-frequency shift} $\pi (\xi)$ by
\begin{equation}
\label{eq:unitary-rep-of-La}
	\pi (\xi) = M_{\omega}T_x \colon L^2 (G) \to L^2 (G),
\end{equation}
where $T_x$ is the \emph{time-shift operator} given by
\begin{equation*}
	\begin{split}
	T_x &\colon L^2 (G) \to L^2 (G)\\
	f(t) &\mapsto f(t-x), \quad t\in G,
	\end{split}
\end{equation*}
and $M_\omega$ is the \emph{modulation operator}, or the \emph{frequency-shift operator}, given by
\begin{equation*}
	\begin{split}
	M_\omega &\colon L^2 (G) \to L^2 (G) \\
	f(t) &\mapsto \omega(t) f(t), \quad t\in G.
	\end{split}
\end{equation*}
We define the \emph{Heisenberg $2$-cocycle}
\begin{equation*}
	\begin{split}
	c \colon (\tfp{G}) \times (\tfp{G}) &\to \mathbb{T} \\
	(\xi_1 , \xi_2) &\mapsto \overline{\omega_2 (x_1)}
	\end{split}
\end{equation*}
for any two elements $\xi_1 = (x_1 , \omega_1), \xi_2 = (x_2 , \omega_2)\in \tfp{G}$. Moreover, we define the associated \emph{symplectic cocycle}
\begin{equation*}
	\begin{split}
	c_s \colon (\tfp{G}) \times (\tfp{G}) &\to \mathbb{T}\\
	(\xi_1 , \xi_2) &\mapsto \overline{\omega_2 (x_1)} \omega_1 (x_2).
	\end{split}
\end{equation*}
for $\xi_1 = (x_1 , \omega_1), \xi_2 = (x_2 , \omega_2)\in \tfp{G}$.
Make particular note of the fact that
\begin{equation*}
	\overline{c(\xi_1 , \xi_2)} = c(-\xi_1,\xi_2) = c(\xi_1, -\xi_2).
\end{equation*}
The $2$-cocycle and the symplectic cocycle are intimately related to time-frequency shifts. Indeed, routine calculations yield the following identities which may be helpful to keep in mind
\begin{equation*}
	\begin{split}
	\pi (\xi_1)\pi(\xi_2) &= c(\xi_1, \xi_2) \pi (\xi_1 + \xi_2)\\
	\pi (\xi_1)\pi (\xi_2) &= c_s (\xi_1 , \xi_2) \pi (\xi_2)\pi (\xi_1) \\
	\pi (\xi_1)^* &= c(\xi_1 , \xi_1) \pi (-\xi_1).
	\end{split}
\end{equation*}
Using the symplectic cocycle $c_s$ we define the \textit{adjoint subgroup of $\La$}, denoted $\Lao$, by
\begin{equation*}
	\Lao : = \{\chi \in \tfp{G} \mid c_s (\chi, \la) = 1 \quad \text{for all $\la \in \La$} \}.
\end{equation*}
It is then clear that $[\pi (\la) , \pi (\chi)] = 0$ for all $\la \in \La$ and all $\chi \in \Lao$. By \cite[p. 234]{JaLe16} we may identify  $\La^{\circ}$ with $((\tfp{G})/\La)^{\widehat{}}$ and we pick the dual measure on $\La^{\circ}$ corresponding to the measure on $(\tfp{G})/\La$ induced from the chosen measure on $\La$. That is, the measures are chosen so that the Plancherel theorem holds with respect to $\Lao$ and $(\tfp{G})/\La$. Since $\La$ is a lattice, it is in particular cocompact, hence it follows that $((\tfp{G})/\La)^{\widehat{}}$ is discrete, from which it follows that $\Lao$ is discrete. But $(\Lao)^\circ \cong \La$ is discrete, from which the analogous argument implies $\Lao$ is also cocompact. Hence $\Lao$ is also a lattice, and we may rightfully call it the \textit{adjoint lattice of $\La$}. Having picked the counting measure on $\La$, the induced measure on $\Lao$ is the counting measure scaled with the constant $s(\La)^{-1}$ \cite[equation (13)]{JaLu18}.

For any function $g \in L^2 (G)$ we may define the \textit{short-time Fourier transform with respect to $g$}. It is the operator 
\begin{equation}
	\begin{split}
	V_g \colon L^2 (G) \to L^2 (\tfp{G}) \\
	V_g f(\xi) = \hs{f}{\pi (\xi)g}.
	\end{split}
\end{equation}
%The short-time Fourier transform will play a major role in the sequel, in particular when defining the (weighted) Feichtinger algebras.
Using the short-time Fourier transform, we define the \textit{Feichtinger algebra} $M^1 (G)$ by
\begin{equation*}
M^1 (G) : = \{f \in L^2 (G) \mid V_f f \in L^1 (\tfp{G})  \}.
\end{equation*}
$M^1 (G)$ becomes a Banach space when equipped with the norm 
\begin{equation*}
\Vert f \Vert_{M^1 (G)} : = \int_{\tfp{G}} | V_g f (\xi) | \dif \xi
\end{equation*}
for some $g \in M^1 (G)\setminus \{0\}$. Indeed it is known that any nonzero $g \in M^1 (G)$ yields an equivalent norm on $M^1 (G)$. We may of course do the same for $\La$. It is however known that when $\La$ is discrete, $M^1 (\La) = \ell^1 (\La)$ with equivalent norms. For proofs of these statements, see for example \cite[Proposition 4.10, Lemma 4.11, Theorem 4.12]{ja18}.

In order to describe smoothness we will need dense subspaces of $M^1 (\La)$ and $M^1 (G)$. To this end we have the following definition. 
\begin{defn}
	\label{def:weight}
	Let $\Sub$ be a second countable LCA group. By a \textit{weight} on $\Sub$ we mean a function $v: \Sub \to [0,\infty)$ satisfying the following conditions:
	\begin{enumerate}
		\item [i)]$v(\xi + \chi) \leq v(\xi) v(\chi)$ for all $\xi, \chi \in \Sub$ (submultiplicativity).
		\item [ii)]$v$ has polynomial growth, i.e. there are $D > 0$ and $s >0$ such that $v(\xi)\leq D(1+d(\xi,0))^s$ for all $\xi \in \Sub$, where $d$ is a translation-invariant metric generating the topology of $\Sub$.
		\item [iii)]$v(\xi) = v(-\xi)$ for all $\xi\in \Sub$ (radial symmetry).
	\end{enumerate}
	A weight $v$ will be called \emph{commutator bounded} if it satisfies the following condition:
	\begin{enumerate}
		\item [iv)] For all $\xi, \chi \in \Sub$ we have $|v(\chi + \xi)-v(\chi)| \leq C_v v(\xi)$, for some $C_v \in [0,\infty)$ depending only on the weight $v$ (commutator bounded condition). 
	\end{enumerate} 
\end{defn}
\begin{rmk}
	\label{rmk:simple-conseqs-of-weight-def}
	If $v(0) = 0$, then for any $\xi \in \Sub$
	\begin{equation*}
	    v(\xi) = v(\xi + 0) \leq v(\xi) v(0)= 0,
	\end{equation*}
	hence the weight $v$ is identically zero. For this reason we will assume in the rest of the article that $v(0) \neq 0$.
	Note then that submultiplicativity of the weight $v$ implies $v(0)\geq 1$. Indeed, by the calculation
	\begin{equation*}
		v(0) = v(0+0) \leq v(0)v(0) = v(0)^2
	\end{equation*}
	we obtain the desired relation by dividing by $v(0)$ on both sides. But by radial symmetry we then have
	\begin{equation*}
		1 \leq v(0) = v(\xi - \xi) \leq v(\xi)v(-\xi) = v(\xi)^2
	\end{equation*}
	for all $\xi \in \Sub$. It follows that $v(\xi)\geq 1$ for all $\xi \in \Sub$.
\end{rmk}
\begin{rmk}
	Note that if $v$ is a weight, so is $v^r$ for $r \in [0,\infty)$. However, even if $v$ is a commutator bounded weight, $v^r$ need not be for $ r \neq 1$.
\end{rmk}
Some of the above assumptions in the definition of a weight are sometimes not present, in order to get a more general version of weights, see for example \cite{gr07}. In the interest of brevity we adopt the definition of weight above.
\begin{defn}
	Let $\Sub$ be a second countable LCA group and let $v$ be a weight on $\Sub$. We then define the \textit{weighted $L^1$-space} $L^1_v (\Sub)$ by 
	\begin{equation}
		L^1_v (\Sub) := \{ f \in L^1 (\Sub) \mid f\cdot v \in L^1 (\Sub)\}.
	\end{equation}
\end{defn}
It is well known that $L^1_v (\Sub)$ is a Banach space with the natural norm, that is, with the norm
	\begin{equation}
		\| f \|_{L^1_v(\Sub)} := \int_{\Sub} |f(\xi)| v(\xi) \dif \xi
	\end{equation}
	for $f \in L^1_v (\Sub)$.

We may then define the relevant subspaces of the Feichtinger algebra.
\begin{defn}
	Let $v: \tfp{G}\to [0,\infty)$ be a weight. We then define the \textit{weighted Feichtinger algebra} $M^1_v (G)$ by
	 \begin{equation}
	 	M^1_v (G) := \{f\in L^2 (G) \mid V_f f \in L^1_v (\tfp{G})   \}.
	 \end{equation}
	 We do the same for $\La$ by restricting weights from $\tfp{G}$ to $\La \subset \tfp{G}$.
\end{defn}
We have the following result from \cite[Theorem 4.1]{fe83-4}.
\begin{prop}
	$M^1_v (G)$ is a Banach space when equipped with the norm
	\begin{equation}
	\| f \|_{M^1_v (G)} := \int_{\tfp{G}} | V_g f (\xi)| v(\xi) \dif \xi,
	\end{equation}
	for some $g \in M^1_v (G)\setminus \{0\}$. Any $g \in M^1_v (G)\setminus \{0\}$ yields an equivalent norm.
\end{prop}

Note that for a weight $v$ of polynomial growth on $\tfp{G}$ the Banach space $M^1_v (G)$ is dense in $M^1 (G)$, because the Schwartz-Bruhat space is dense in $M^1(G)$ by \cite{Po80} and by Osborne's characterization of the Schwartz-Bruhat space \cite{Os75}. 

In \cref{sec:Link-Gab-an} we will link the $QC^k$-structure statements for Heisenberg modules of \cref{section:modulation-spaces-as-smooth-modules} with the study of Gabor frames. To this end we introduce the relevant concepts from Gabor analysis now.
\begin{defn}
	\label{Definition: Gabor system}
	A \textit{Gabor system} $\mathcal{G}( g; \Lambda)$ is a collection of time-frequency shifts of a function $g$ of the form $\{ \pi(\lambda)g \mid \lambda \in \Lambda\}$. We call it a \emph{Gabor frame for $L^2 (G )$} if it is a frame for the Hilbert space $L^2 (G)$, that is, if the following inequalities are satisfied for all $f\in L^2 (G)$
	\begin{equation}
	\label{Equation: Frame definition}
	C \vert \vert f \vert \vert_2^2 \leq \sum_{\lambda \in \Lambda} \vert \hs{f}{\pi(\la)g}  \vert^2 \leq D \vert \vert f \vert \vert_2^2 ,
	\end{equation}
	for some $0<C \leq D<\infty$. If $C=D=1$, we call $\mathcal{G}( g; \Lambda)$ a \emph{Parseval Gabor frame}. If only the upper frame bound is satisfied, we say  $\mathcal{G}( g; \Lambda)$ is a \emph{Bessel system}, and the function $g$ is called a \emph{Bessel vector for $\La$}.
	
	Extending to the case where we have functions $g_1 ,\ldots , g_l \in L^2 (G)$, we define a multiwindow Gabor system by $\mathcal{G} (g_1 , \ldots , g_l; \Lambda) := \mathcal{G}(g_1 ; \Lambda) \cup \cdots \cup \mathcal{G}(g_l ; \Lambda)$. We call it a \emph{multiwindow Gabor frame for $L^2 (G)$} if there exist constants $0<C \leq D<\infty$ such that
	\begin{equation}
	\label{Equation: Frame condition multiwindow}
	C \vert \vert f \vert \vert_2^2 \leq \sum_{i=1}^l \sum_{\lambda \in \Lambda} \vert \hs{f}{\pi (\la)g_i} |^2 \leq D \vert \vert f \vert \vert_2^2 
	\end{equation}
	for all $f\in L^2 (G)$. Again, if $C=D=1$ we call $\mathcal{G} ( g_1 , \ldots , g_l ; \Lambda)$ a \emph{Parseval multiwindow Gabor frame}. If only the upper frame bound is satisfied, we say $\mathcal{G} ( g_1 , \ldots , g_l ; \Lambda)$ is a \emph{Bessel system}, and the functions $\{g_1, \ldots , g_l \}$ are called \emph{Bessel vectors for $\La$}.
\end{defn}
Intimately related to Bessel systems $\GS(g;\La)$ are the \textit{coefficient mapping}
\begin{equation}
C_{g, \Lambda}: L^2 (G) \to \ell^2 (\Lambda), \quad f \mapsto \{\hs{f}{\pi (\la)g}\}_{\lambda \in \La} ,
\end{equation}
and the \textit{synthesis mapping}
\begin{equation}
D_{g, \Lambda}: \ell^2 (\Lambda) \to L^2 (G), \quad \{c_{\lambda}\}_{\lambda} \mapsto \sum_{\lambda \in \Lambda} c_{\lambda} \pi (\lambda) g
\end{equation}
A straightforward calculation shows that $D_{g, \Lambda} = C_{g, \Lambda}^*$. These allow us to define the following operator.
\begin{defn}
	For a Bessel system $\mathcal{G}( g ; \Lambda)$ we define the \emph{Gabor frame operator} $S_{g, \Lambda}$ by
	\begin{equation}
	S_{g, \Lambda}: L^2 (G) \to L^2 (G), \quad S_{g, \Lambda} = D_{g, \Lambda} \circ C_{g, \Lambda}.
	\end{equation}
	Likewise, given a multiwindow Bessel system $\mathcal{G} ( g_1 , \ldots , g_l ; \Lambda)$, we define the \emph{multiwindow Gabor frame operator} $S_{g_1, \ldots, g_l , \Lambda}$ by
	\begin{equation}
	S_{g_1, \ldots, g_l , \Lambda} : L^2 (G) \to L^2 (G) , \quad S_{g_1, \ldots, g_l , \Lambda} = \sum_{i=1}^l S_{g_i , \Lambda}.
	\end{equation}
\end{defn}
Note that boundedness of the (multiwindow) Gabor frame operator is guaranteed by the upper norm bounds in \eqref{Equation: Frame definition} and \eqref{Equation: Frame condition multiwindow}. If $\GS(g_1, \ldots , g_l ; \La)$ is a frame, the corresponding lower bound guarantees that the (multiwindow) Gabor frame operator is invertible. Also, since $S_{g, \Lambda} = C_{g, \Lambda}^* \circ C_{g, \Lambda}$, the Gabor frame operator is positive and thus the multiwindow Gabor frame operator is positive, too. Hence for a Gabor frame $\mathcal{G} (g; \Lambda)$ (resp. a multiwindow Gabor frame $\mathcal{G} (g_1 , \ldots g_l ; \Lambda)$) the corresponding Gabor frame operator $S_{g, \Lambda}$ (resp. multiwindow Gabor frame operator $S_{g_1, \ldots g_l , \Lambda}$) is a bounded, positive, and invertible operator. 
Indeed, the converse is true also, and we will need the following well-known result.
\begin{prop}
	Let $\La$ be a lattice in $\tfp{G}$ where $G$ is a second countable LCA group and let $\{g_1, \ldots , g_l \}$ be Bessel vectors for $\La$. Then $\GS (g_1 , \ldots , g_l; \La)$ is a frame for $L^2 (G)$ if and only if the frame operator $S_{g_1, \ldots, g_l , \Lambda}$ is invertible.
\end{prop}

\section{Twisted group $C^*$-algebras and $QC^k$-structures}
\label{Section:twisted-group-algebras-and-Ck}
\subsection{Twisted group $C^*$-algebras}
\label{section:twisted-group-algebras}
We proceed to introduce the relevant Banach $*$-algebras and $C^*$-algebras. As in \cref{sec:Gab-An-LCA-weight-Feich-algs} we let $G$ denote a second countable locally compact abelian group and let $\La \subset \tfp{G}$ be a lattice. Furthermore, $v$ will be a weight on $\tfp{G}$, and $c$ denotes the Heisenberg $2$-cocycle. Indeed, in the rest of the paper $c$ will denote this $2$-cocycle. We then wish to study the \textit{$v$-weighted $c$-twisted group algebra} $\ell^1_v (\La,c)$. This is the space $\ell^1_v (\La)$ equipped with $c$-twisted convolution
\begin{equation}
	a_1 *_c a_2 (\la) = \sum_{\la' \in \La} a_1 (\la') a_2 (\la - \la') c(\la', \la - \la') 
\end{equation}
and $c$-twisted involution
\begin{equation}
	a^* (\la) = c(\la,\la)\overline{a(-\la)}
\end{equation}
for $a,a_1,a_2 \in \ell^1_v (\La)$ and $\la \in \La$. The unweighted space, that is, with weight $v=1$, will be denoted $\ell^1 (\La,c)$. 
\begin{rmk}
    We will sometimes suppress the notation $*_c$ and just write $a_1 a_2$ instead of $a_1 *_c a_2$.
\end{rmk}
The following result is well known. %For the reader's convenience we include a proof that the corresponding statement for the weighted algebra is true.
\begin{prop}
	$\ell^1_v (\La,c)$ is a Banach $*$-algebra. 
\end{prop}
%\begin{proof}
%	Since we already know that $\ell^1 (\La,c)$ is a Banach $*$-algebra, it suffices to prove that the involution is isometric and that the norm is submultiplicative when accounting for the weight.
	
%	Since the weight $v$ is radially symmetric we have for any $a \in \ell^1_v (\La,c)$
%	\begin{equation*}
%		\begin{split}
%		\| a^*\|_{\ell^1_v} &= \sum_{\la \in \La} |a^* (\la) | v(\la) = \sum_{\la \in \La} | c(\la,\la) \overline{a(-\la)}| v(\la) \\
%		&= \sum_{\la\in \La} |a(-\la)| v(\la) = \sum_{\la \in \La} |a (\la)| v(-\la)\\
%		&= \sum_{\la \in \La} | a(\la)| v(\la) = \| a \|_{\ell^1_v},
%		\end{split}
%	\end{equation*}
%	so it follows that the involution is an isometry. With $a_1, a_2 \in \ell^1_v (\La,c)$ it follows by submultiplicativity of the weight $v$ that
%	\begin{equation*}
%	\begin{split}
%	\| a_1 \natural a_2 \|_{\ell^1_v} &= \sum_{\la \in \La} | a_1 \natural a_2 (\la) | v(\la)\\
%	& = \sum_{\la \in \La} | \sum_{\mu \in \La} a_1 (\mu) a_2 (\la - \mu) c(\mu,\la-\mu) | v(\la) \\
%	&\leq \sum_{\la \in \La} \sum_{\mu \in \La} | a_1 (\mu) | a_2 (\la - \mu)| | c(\mu, \la - \mu)| v(\la) \\
%	&= \sum_{mu\in \La} \sum_{\la \in \La} |a_1 (\mu)| |a_2 (\la -mu)| v(\la) \\
%	&= \sum_{\mu\in \La} \sum_{\la \in \La} |a_1 (\mu)| |a_2 (\la)| v(\la + \mu) \\
%	&\leq  \sum_{\mu\in \La} \sum_{\la \in \La} |a_1 (\mu)| |a_2 (\la)| v(\la)v(\mu) \\
%	&= \sum_{\mu \in \La} |a_1 (\mu)|v(\mu) \sum_{\la \in \La} |a_2 (\la)| v(\la) \\
%	&= \| a_1 \|_{\ell^1_v} \|a_2\|_{\ell_v^1}, 
%	\end{split}
%	\end{equation*}
%	which shows that the norm is submultiplicative.
%\end{proof}
We may do the same for $\ell^1_v (\Lao)$ to make $\ell^1_v (\Lao, \overline{c})$. Note the conjugate cocycle. 

There is a natural way of associating to $\ell_v^1 (\La,c)$ a $C^*$-algebra. Indeed, we do the procedure for $\ell^1 (\La,c)$ to complete it to a $C^*$-algebra, and it will be clear that by density of $\ell_v^1 (\La,c)$ in $\ell^1 (\La,c)$ we would obtain the same $C^*$-algebra if we were to do the same procedure with $\ell_v^1 (\La,c)$. 
The procedure is as follows. We have a $c$-projective unitary representation of $\La$ on $L^2 (G)$ via \eqref{eq:unitary-rep-of-La}. This gives a nondegenerate $c$-projective $*$-representation of $\ell^1 (\La)$, or equivalently, a nondegenerate $*$-representation of $\ell^1 (\La,c)$, on $L^2 (G)$ by setting
\begin{equation*}
	a \cdot f = \sum_{\la \in \La} a(\la) \pi (\la)f
\end{equation*}
for $f \in L^2 (G)$ and $a \in \ell^1 (\La)$. This representation is faithful \cite{Ri88}. We thus obtain a $C^*$-algebra %$C^* (\La,c)$ 
by completing $\ell^1 (\La,c)$ in the norm $ \| a \|_{\mathbb{B}(L^2 (G))}$ for $a \in \ell^1 (\La,c)$. %This is known as the \textit{reduced $c$-twisted group $C^*$-algebra of $\La$}. 
But $\La$ is an abelian group, hence it is amenable. So the reduced and full $c$-twisted group $C^*$-algebras of $\La$ coincide, and so we may denote the common $C^*$-completion of $\ell^1 (\La,c)$ by $C^* (\La,c)$, and denote the norm by $\| \cdot \|_{C^*}$. We refer to this $C^*$-algebra as the \textit{$c$-twisted group $C^*$-algebra of $\La$}. Since $\ell_v^1 (\La,c)$ is dense in $\ell^1 (\La,c)$ and $\| \cdot \|_{\ell^1_v} \geq \| \cdot \|_{\ell^1} \geq \| \cdot \|_{C^*}$, we would obtain the same $C^*$-algebra by doing the procedure with $\ell_v^1 (\La,c)$. We do the same procedure for $\ell_v^1 (\Lao,\overline{c})$ and obtain $C^* (\Lao, \overline{c})$. The $C^*$-algebras $C^* (\La,c)$ and $C^* (\Lao, \overline{c})$ are closely related. Indeed, they are Morita equivalent, which we will discuss in section \cref{section:modulation-spaces-as-modules}, and will have use for in \cref{section:modulation-spaces-as-smooth-modules}.

To show that weighted Feichtinger algebras are examples of $QC^k$-modules in \cref{section:modulation-spaces-as-smooth-modules} we will show how certain module frames implement said $QC^k$-structure. It will then be important that the module frames are suitably regular. To guarantee this we need the following important result from \cite{GrLe04}.
\begin{prop}\label{prop:GrLe-spec-inv}
	Let $G$ be a second countable LCA group, let $\La \subset \tfp{G}$ a discrete subset, and let $v$ be a weight on $\tfp{G}$. Then $\ell^1_v (\La,c)$ is spectral invariant in $C^* (\La ,c)$.
\end{prop}
\subsection{Weighted Feichtinger algebras as modules}
\label{section:modulation-spaces-as-modules}
In order to get the desired modules we will need the following result, see \cite[Proposition 5.1, Proposition 5.2]{FeGr89}. The arguments for $G=\mathbb{R}$ extend in a straightforward way to the general case.
\begin{prop}
	\label{prop:cty-and-sampling-props}
Let $\La \subset \tfp{G}$ be a lattice and let $v$ be a weight on $\tfp{G}$. Then the following hold:
\begin{enumerate}
	\item [i)] Let $\la \in \La$ and $f \in M_v^1 (G)$. Then $\pi (\la)f \in M_v^1 (G)$ and
	\begin{equation*}
		\| \pi (\la)f \|_{M^1_v} \leq v(\la) \| f\|_{M^1_v}.
	\end{equation*}
	\item [ii)] If $a \in \ell_v^1 (\La)$ and $f \in M_v^1 (G)$, then $	\sum_{ \la \in \La} a(\la)\pi(\la)f \in M_v^1 (G)$ and
	\begin{equation*}
	\| \sum_{ \la \in \La} a(\la)\pi(\la)f \|_{M_v^1} \leq C \| a \|_{\ell_v^1} \| f \|_{M_v^1}
	\end{equation*}
	for some $C > 0 $ independent of $a$ and $f$.
	\item [iii)] If $f,g \in M_v^1 (G)$, then $(\hs{f}{\pi(\la)g})_{\la \in \La} \in \ell_v^1 (\La)$.
\end{enumerate}	
\end{prop}
We may now give $M_v^1 (G)$ a left Banach $\ell_v^1 (\La,c)$-module structure by defining
\begin{equation}\label{eq:left-mod-action}
	a \cdot f = \sum_{\la \in \La} a(\la) \pi(\la)f
\end{equation}
for $a \in \ell_v^1 (\La,c)$ and $f \in M_v^1 (G)$. We may turn $M_v^1 (G)$ into an inner product module over $\ell_v^1 (\La,c)$ by defining
\begin{equation}\label{eq:left-inner-product}
	\lhs{f}{g} = (\hs{f}{\pi (\la)g})_{\la\in \La},
\end{equation}
for $f,g \in M_v^1 (G)$. Here $\lhs{\cdot}{\cdot}$ is the $\ell_v^1 (\La,c)$-valued inner product. That the module action is continuous, and that the module action and the inner product are well-defined follows from \cref{prop:cty-and-sampling-props}. We likewise get a right $\ell_v^1 (\Lao,\overline{c})$-inner product module structure on $M_v^1 (G)$ by setting
\begin{equation}\label{eq:right-mod-action}
	f \cdot b = \sum_{ \lao \in \Lao} b(\lao) \pi (\lao)^* f
\end{equation}
for $f \in M_v^1 (G)$ and $b \in \ell_v^1 (\Lao , \overline{c})$, and
\begin{equation}\label{eq:right-inner-product}
	\rhs{f}{g} = ( \hs{\pi (\lao)g}{f})_{\lao \in \Lao}
\end{equation}
for $f,g \in M_v^1 (G)$. These are also well-defined by \cref{prop:cty-and-sampling-props}. Actually, with the above defined actions and inner products $M_v^1 (G)$ becomes a pre-equivalence bimodule between $\ell_v^1 (\La,c)$ and $\ell_v^1 (\Lao,\overline{c})$. This verification was done for the Schwartz-Bruhat case in \cite[Theorem 2.15]{Ri88} and for the Feichtinger algebra case in \cite[Theorem 3.13]{Lu09}. We may complete $ M_v^1 (G)$ in the Hilbert $C^*$-module norm coming from $C^* (\La,c)$ (or equivalently the norm from $C^* (\Lao,\overline{c})$) to obtain a $C^* (\La,c)$-$C^* (\Lao,\overline{c})$-equivalence bimodule, which we will denote by $\heis{G}{\La}$. Such modules are known in the literature as \emph{Heisenberg modules}.

The $C^*$-algebra $C^* (\La,c)$ has a very useful property which we will have great need for in \cref{section:smooth-structures-algebra}. Indeed, it can be continuously embedded in $\ell^2 (\La)$. Likewise, the Heisenberg module $\heis{G}{\La}$ can be continuously embedded in $L^2 (G)$. These statements can be proved by ways of localization as in \cite{AuEn19}. However, since we are working exclusively with lattices in phase space, we use a different and simpler proof.
\begin{prop}
	\label{prop:algs-and-mod-are-fct-spaces}
	$C^* (\La,c) \hookrightarrow \ell^2 (\La)$, $C^* (\Lao,\overline{c}) \hookrightarrow \ell^2 (\Lao)$, and $\heis{G}{\La} \hookrightarrow L^2 (G)$ continuously.
\end{prop}
\begin{proof}
	Since $\La$ is discrete, $C^*(\La,c)$ is equipped with a finite faithful trace \cite[p. 951]{BeOm18} given by the continuous extension of evaluation in $0 \in \La$, that is, the continuous extension of the trace
	\begin{equation*}
	\begin{split}
	\tr: \ell^1 (\La,c) &\to \C \\
	a &\mapsto a(0).
	\end{split}
	\end{equation*}
	A straightforward calculation will show that
	\begin{equation*}
	\tr (aa^*)= \Vert a \Vert_2^2.
	\end{equation*}
	for all $a \in \ell^1 (\La , c)$.
	Since $\tr$ is continuous, it follows that if $(a_n)_n \subset \ell^1 (\La,c)$ is a Cauchy sequence in $C^* (\La,c)$-norm, then 
	\begin{equation*}
		0 = \lim_{m,n\to \infty} \tr ((a_m - a_n)(a_m-a_n)^* ) = \Vert a_m - a_n \Vert_2^2.
	\end{equation*}
	As a result we can define a map $\iota:C^* (\La,c) \to \ell^2 (\La)$ as the continuous extension of the corresponding identity map on $\ell^1 (\La)$. Using that $\iota(\overline{\ell^1 (\La)}) \subset \overline{\iota (\ell^1 (\La))}$, it follows that $C^* (\La,c)\subset \ell^2 (\La)$. We can of course do the same for $C^* (\Lao,\overline{c})$.
	
	In very much the same vein, a straightforward calculation will show that
	\begin{equation*}
		\tr (\lhs{g}{g}) = \| g \|_2^2
	\end{equation*}
	for all $g \in M^1 (G)$. 
	Now let $(f_n)_n \subset M^1 (G)$ be a Cauchy sequence in $\heis{G}{\La}$-norm. Then once again, since $\tr$ is continuous, we have
	\begin{equation*}
		0 = \lim_{m,n\to \infty} \tr(\blangle f_m - f_n, f_m - f_n\rangle) = \lim_{m,n\to \infty} \| f_m - f_n \|_2^2.
	\end{equation*}
	We may then once again define a map $\iota : \heis{G}{\La} \to L^2 (G)$ as the continuous extension of the identity map on $M^1 (G)$. Once again using that $\iota (\overline{M^1 (G)})\subset \overline{\iota(M^1 (G))}$ it follows that $\heis{G}{\La}\subset L^2 (G)$.
\end{proof}

\begin{exmp}[The noncommutative $2$-torus]
\label{ex:Heis-mod-for-R2}
	We look at how we obtain the noncommutative $2$-torus from the above constructions and how the weighted Feichtinger algebras $M_v^1 (\R)$ can be completed to Hilbert $C^*$-modules. For details we refer the reader to \cite{Lu09} where this is done in depth for more general noncommutative $2d$-tori, $d \in \N$.
	
	Let $(x, \omega)\in \R \times \widehat{\R} \cong \R^2$. On $L^2 (\R)$ the time shift operator $T_x$ is then
	\begin{equation*}
	T_x f(t) = f(t-x), \quad t\in \R,
	\end{equation*}
	and the modulation operator $M_{\omega}$ is
	\begin{equation*}
	M_\omega f(t) = e^{2\pi i \omega t} f(t), \quad t\in \R,
	\end{equation*}
	for $f \in L^2 (\R)$. The time-frequency shift $\pi (x,\omega)$ is then
	\begin{equation*}
	\pi (x,\omega) f(t) = e^{2\pi i \omega t} f(t-x), \quad t\in \R,
	\end{equation*}
	for $f \in L^2 (\R)$. Moreover, the Heisenberg $2$-cocycle is given by
	\begin{equation*}
	c((x,\omega),(y,\eta)) = e^{-2\pi i \eta x }.
	\end{equation*}
	Now let $\La \subset \R^2$ be a lattice and let $v$ be a weight on $\R^2$. 
	As before we get a faithful representation of $\ell_v^1 (\La, c)$ on $L^2 (\R)$ by
	\begin{equation*}
	a\cdot f = \sum_{\la \in \La} a(\la) \pi(\la)f
	\end{equation*}
	for $f \in L^2 (\R)$ and $a \in \ell_v^1 (\La, c)$. Completing $\ell_v^1 (\La, c)$ in the induced operator norm we obtain a $C^*$-algebra $C^* (\La,c)$, which is also known as the noncommutative $2$-torus. The usual noncommutativity parameter $\theta$ of e.g. \cite{Ri81} is determined by the lattice. In particular, $\La = L \Z^2$ for some $L \in GL(\R^2)$, and then $\theta = \det L$. By ways of \eqref{eq:left-mod-action} and \eqref{eq:left-inner-product} we complete $M^1_v (\R)$ to a Heisenberg module $\heis{\R}{\La}$ over $C^* (\La,c)$. We may indeed do the same for $\ell^1_v (\Lao,\overline{c})$ and get a right Hilbert module structure by ways of \eqref{eq:right-mod-action} and \eqref{eq:right-inner-product}. Then $\heis{\R}{\La}$ becomes a $C^* (\La,c)$-$C^* (\Lao,\overline{c})$-equivalence bimodule with $M^1_v (\R)$ as an $\ell^1_v (\La,c)$-$\ell^1_v (\Lao,\overline{c})$-pre-equivalence bimodule.
\end{exmp}

\subsection{Smooth structures on twisted group $C^*$-algebras}
\label{section:smooth-structures-algebra}
At last we can make precise a $QC^k$-structure on twisted group $C^*$-algebras of lattices in phase space. To do this we introduce the relevant spectral triples. 

We observed in \cref{prop:algs-and-mod-are-fct-spaces} that $C^* (\La,c)$ embeds continuously into $\ell^2 (\La)$. Indeed, we obtain $\ell^2 (\La)$ if we complete $C^*(\La,c)$ in the inner product given by
\begin{equation*}
	\hs{a_1}{a_2} = \tr (a_1^* a_2)
\end{equation*}
for $a_1,a_2\in A$. 
This is immediate by the proof of \cref{prop:algs-and-mod-are-fct-spaces}. There is then a continuous action of $C^* (\La,c)$ on $\ell^2 (\La)$ from the left, which is the continuous extension of the left action of $C^* (\La,c)$ on itself by multiplication.
If $v$ is a commutator bounded weight on $\tfp{G}$ and $n \in \N$ we may then consider the even spectral triple given by
\begin{equation}
	(\ell_{v^n}^1 (\La,c), \ell^2 (\La)\oplus \ell^2 (\La), D),
\end{equation}
where $D$ is the (unbounded) selfadjoint operator given by
\begin{equation} 
	D = 
	\begin{pmatrix}
	0 & v \\
	v & 0
	\end{pmatrix}.
\end{equation}
However, we may consider a more general situation.
\begin{defn}
	\label{def:spec-trip-comp-fct}
	Let $v$ be a weight on $\tfp{G}$ (not necessarily commutator bounded), let $\La \subset \tfp{G}$ be a lattice, and let $f: [0,\infty) \to [0,\infty)$ be a function. We say $f$ is \emph{spectral triple compatible for $v$ with respect to $\La$} if the following conditions are satisfied:
	\begin{enumerate}
		%\item [i)] $f(v)(\la)\geq 0$ for all $\la \in \La$.
		\item [i)] There is a constant $\Cdif\in [0,\infty)$ such that
		\begin{equation*}
		|f(v) (\la+\mu) - f(v)(\la)| \leq \Cdif f(v)(\mu)
		\end{equation*}
		for all $\la, \mu \in \La$.
		\item [ii)] There is a constant $\Csm \in [0,\infty)$ such that
		\begin{equation*}
		f(v)(\la+\mu) \leq \Csm f(v)(\la)\cdot f(v)(\mu)
		\end{equation*}
		for all $\la, \mu \in \La$.
		\item [iii)] There is a constant $\Cgr\in [0,\infty)$ such that
		\begin{equation*}
		f(v)(\la) \leq \Cgr v(\la)
		\end{equation*}
		for all $\la \in \La$.
	\end{enumerate}
\end{defn}
\begin{rmk}
    We interpret $f(v)(\la)$ as $f(v(\la))$.
\end{rmk}
\begin{rmk}
	Note that condition iii) of \cref{def:spec-trip-comp-fct} implies that for any $q \in [1,\infty]$ there is a constant $C_q \in [0,\infty)$, depending only on $q$, such that
	\begin{equation*}
		\| a\|_{\ell^q_{f(v)}(\La)} \leq C_q \|a \|_{\ell^q_{v}(\La)}
	\end{equation*}
	for all $a \in \ell^q_{f(v)}(\La)$. Moreover, since $f$ may have zeros, $\ell^q_{f(v)}(\La)$ is in general not a Banach space.
\end{rmk}
\begin{rmk}
	The subscripts $\mathrm{dif}$, $\mathrm{sm}$ and $\mathrm{gr}$ on the constants in \cref{def:spec-trip-comp-fct} are chosen so that in subsequent calculations it will be easier to understand which properties of $f$ are being invoked. The subscript $\mathrm{dif}$ reflects that it expresses a bound on a difference, the subscript $\mathrm{sm}$ reflects a form of submultiplicativity, and $\mathrm{gr}$ reflects a growth condition.
\end{rmk}
We observe that for any weight $v$, the set of spectral triple compatible functions with respect to any lattice is nonempty. Indeed, nonnegative constant functions will satisfy the conditions of \cref{def:spec-trip-comp-fct} regardless of the choice of lattice and weight. More interestingly, if $v$ is a commutator bounded weight, we see that the function $f: [0,\infty) \to [0,\infty)$ given by $f(t) = t$ is spectral triple compatible for $v$ with respect to any lattice.

For a lattice $\La \subset \tfp{G}$, a weight $v$ on $\tfp{G}$, and a spectral triple compatible function $f$ for $v$ with respect to $\La$, we then consider the (unbounded) selfadjoint operator $D$ on $\ell^2 (\La) \oplus \ell^2 (\La)$ given by
\begin{equation}
\label{eq:def-D-spectral-op}
	D= \begin{pmatrix}
	0 & f(v) \\ 
	f(v) & 0
	\end{pmatrix}.
\end{equation}
We have the following result.
\begin{prop}
	\label{prop:n-geq-k+1-gives-QCk}
	Let $v$ be a weight on $\tfp{G}$, let $\La$ be a lattice in $\tfp{G}$, let $f$ be a spectral triple compatible function for $v$ with respect to $\La$, and let $D$ be defined by \eqref{eq:def-D-spectral-op}. Then
	$(\ell_{v^n}^1 (\La,c), \ell^2 (\La)\oplus \ell^2 (\La), D)$ is an even spectral triple for $C^*(\La,c)$ whenever $n\geq 1$, and $\ell_{v^n}^1 (\La,c) \subset QC^* (\La,c)_k$ for $n\geq k+1$. In other words, if $n\geq k+1$ then the spectral triple is quantum $C^k$. 
\end{prop}
\begin{proof}
	We begin by verifying that $(\ell_{v^n}^1 (\La,c), \ell^2 (\La)\oplus \ell^2 (\La), D)$ is an even spectral triple for $C^*(\La,c)$ when $n\geq 1$.
	Note first that $\Dom (D)$ is given by 
	\begin{equation*}
	\Dom (D) = \{ (b , b')^T \in \ell^2 (\La) \oplus \ell^2 (\La) \mid f(v) b,f(v)b' \in \ell^2 (\La)  \}.
	\end{equation*}
    
    Throughout the proof, will do the calculations as if the action of $\ell_{v^n}^1 (\La,c)$ on $\ell^2 (\La)$ is by $c$-twisted convolution, denoted $*_c$. This is technically only true on a dense subspace (for example $\ell^1_v (\La)\subset \ell^2 (\La)$), but the actual action is the continuous extension of $c$-twisted convolution. Due to the many conditions we need to check in this proof, we will not make an effort to specify that the elements of $\ell^2 (\La)$ are such that the action of $\ell^1 (\La,c)$ on them is given by $c$-twisted convolution. Rather we will just assume this for simplicity, and
    it will be clear from the calculations that the results go through with the usual extension by density arguments.
    
	To see that $a \cdot \Dom (D) \subset \Dom (D)$ for all $a \in \ell_{v^n}^1 (\La,c)$, let $(b,b')^T \in \Dom (D)$. 
	%Then, since the action of $\ell_{v^n}^1 (\La,c)$ on $\ell^2 (\La)$ is by $c$-twisted convolution, and 
	Due to the form of $D$ it suffices to show $f(v) \cdot (a \ast_c b)\in \ell^2 (\La)$. We have the following
	\begin{align*}
		\|f(v)\cdot (a \ast_c b) \|_{\ell^2(\La)}^2 & = \sum_{\la \in \La} f(v)(\la)^2 \lvert \sum_{\mu \in \La} a(\mu)b(\la-\mu) c(\mu, \la - \mu) \lvert^2 \displaybreak[0] \\
		&\leq \sum_{ \la \in \La} \sum_{\mu\in \La} f(v)(\la)^2 |a(\mu)|^2 |b(\la - \mu)|^2 |c(\mu,\la - \mu)|^2 \displaybreak[0] \\
		&= \sum_{ \la \in \La} \sum_{\mu\in \La} f(v)(\la + \mu)^2 |a (\mu)|^2 |b(\la)|^2 \displaybreak[0] \\
		&\leq \Csm^2 \sum_{ \la \in \La} \sum_{\mu\in \La} f(v)(\la)^2f(v)(\mu)^2 | a(\mu)|^2 |b(\la)|^2  \displaybreak[0]\\
		&= \Csm^2 \sum_{ \mu \in \La} f(v)(\mu)^2 |a(\mu)|^2  \sum_{\la \in \La} f(v)(\la)^2 |b(\la)|^2  \displaybreak[0]\\
		&= \Csm^2 \| f(v)a\|^2_{\ell^2 (\La)} \|f(v) b\|^2_{\ell^2 (\La)} \\
		&\leq \Csm^2  \| f(v) a\|^2_{\ell^1(\La)} \| f(v) b\|^2_{\ell^2 (\La)} \displaybreak[0]\\
		&= \Csm^2  \| a\|^2_{\ell^1_{f(v)}(\La)} \| f(v) b\|^2_{\ell^2 (\La)} \displaybreak[0]\\
		&\leq \Csm^2 \Cgr^2 \| a \|^2_{\ell^1_{v}(\La)}\| f(v) b\|^2_{\ell^2 (\La)}.
	\end{align*}
Here we used $|c(\mu,\la )| = 1$ for all $\mu,\la \in \La$ and $\| f(v)a\|_{\ell^2 (\La)} \leq \| f(v) a\|_{\ell^1(\La)}$. By the calculation it follows that $a \Dom (D) \subset \Dom (D)$ for all $a \in \ell_{v^n}^1 (\La,c)$.

To show that $[D,a]$ extends to a bounded operator on $\ell^2 (\La)$ for all $a \in \ell^1_{v^n}(\La)$, note that for $ (b , b')^T \in \ell^2 (\La)\oplus \ell^2 (\La)$ we have
\begin{equation*}
	[D,a] \begin{pmatrix}
	b \\b'
	\end{pmatrix}
	= \begin{pmatrix}
	f(v)\cdot (a\ast_c b') - a \ast_c (f(v)\cdot b')\\
	f(v) \cdot (a \ast_c b) - a \ast_c (f(v) \cdot b)
	\end{pmatrix}.
\end{equation*} 
Hence it suffices to show that there is $C \geq 0$ such that $\|f(v)\cdot (a\ast_c b) - a \ast_c (f(v)\cdot b)\|_{\ell^2 (\La)} \leq C \|b\|_{\ell^2 (\La)}$ for all $b \in \ell^2 (\La)$. Using $|c(\mu,\la )| = 1$ for all $\mu,\la \in \La$ and  \cref{def:spec-trip-comp-fct} we then have
\begin{align*}
	\|f(v)\cdot (a&\ast_c b) - a \ast_c (f(v)\cdot b)\|^2_{\ell^2 (\La)} \\
	&= \sum_{\la \in \La}\lvert \sum_{ \mu \in \La} f(v)(\la)a(\mu)b(\la - \mu) c(\mu , \la - \mu) - a(\mu)b(\la - \mu)f(v)(\la - \mu)c(\mu,\la-\mu) \rvert^2 \displaybreak[0] \\
	&\leq \sum_{ \la \in \La} \sum_{\mu \in \La} \lvert f(v)(\la) a(\mu) b(\la-\mu) c(\mu,\la - \mu) - a(\mu) b(\la-\mu) f(v)(\la-\mu)c(\mu,\la - \mu)\rvert^2 \displaybreak[0]\\
	&= \sum_{ \la \in \La} \sum_{\mu \in \La} | a(\mu)|^2 |b(\la-\mu)|^2 |c(\mu,\la - \mu)|^2 |f(v)(\la)-f(v)(\la-\mu)|^2 \displaybreak[0]\\
	&= \sum_{ \la \in \La} \sum_{\mu \in \La} |a(\mu)|^2 |b(\la)|^2 |f(v)(\la + \mu) - f(v)(\la)|^2 \displaybreak[0] \\
	&\leq\sum_{ \la \in \La} \sum_{\mu \in \La} |a(\mu)|^2 |b(\la)|^2 \Cdif^2 f(v)(\mu)^2 \displaybreak[0]\\
	&= \Cdif^2 \sum_{\mu \in \La} |a(\mu)|^2 f(v)(\mu)^2 \sum_{ \la \in \La} |b(\la)|^2 \displaybreak[0]\\
	&= \Cdif^2 \| a f(v)\|^2_{\ell^2(\La)} \| b\|^2_{\ell^2 (\La)} \displaybreak[0]\\
	&\leq \Cdif^2 \| a f(v)\|^2_{\ell^1(\La)} \| b\|^2_{\ell^2 (\La)} \displaybreak[0]\\
	&= \Cdif^2 \| a \|^2_{\ell^1_{f(v)}(\La)} \| b\|^2_{\ell^2 (\La)} \displaybreak[0]\\
	&\leq \Cgr^2 \Cdif^2 \| a \|^2_{\ell^1_{v}(\La)} \| b\|^2_{\ell^2 (\La)},
\end{align*}
where we have used $\| a f(v)\|_{\ell^2(\La)}  \leq \| a f(v) \|_{\ell^1(\La)}$. It follows that $[D,a]$ extends to a bounded operator on $\ell^2 (\La)$. Lastly, we need to verify that for any $a \in \ell^1_{v^{n}}(\La,c)$, $a (1+D^2)^{-1/2}$ extends to a compact operator on $\ell^2 (\La)$. Since $D$ is just a multiplication operator, we see that $(1+D^2)^{-1/2}$ is just the multiplication operator
\begin{equation*}
	(1+D^2)^{-1/2} = \begin{pmatrix}
	(1+f(v)^2)^{-1/2} & 0 \\
	0 & (1+f(v)^2)^{-1/2}
	\end{pmatrix}.
\end{equation*}
For $(b , b')^T \in \ell^2 (\La) \oplus \ell^2 (\La)$ and $a \in \ell^1_{v^n}(\La,c)$ we have
\begin{equation*}
	a(1+D^2)^{-1/2} \begin{pmatrix}
	b \\ b'
	\end{pmatrix}
	= \begin{pmatrix}
	a \ast_c ((1+f(v)^{2})^{-1/2}\cdot b) \\
	a \ast_c ((1+f(v)^{2})^{-1/2} \cdot b')
	\end{pmatrix},
\end{equation*}
hence we once again verify the statement in one component. By Young's inequality we obtain
\begin{equation*}
	\begin{split}
	\| a \ast_c ((1+f(v)^{2})^{-1/2} b)\|_{\ell^2(\La)}^2 &\leq \| a\|^2_{\ell^1 (\La)} \| (1+f(v)^{2})^{-1/2} b\|^2_{\ell^2 (\La)}\\
	&\leq \| a\|^2_{\ell^1_{v}(\La)} \|b\|^2_{\ell^2 (\La)},
	\end{split}
\end{equation*}
since $\| a\|_{\ell^1 (\La)} \leq \| a\|_{\ell^1_{v}(\La)}$ and $\| (1+f(v)^{2})^{-1/2} b\|^2_{\ell^2 (\La)} \leq \|b\|^2_{\ell^2 (\La)}$ as $f(v)^2\geq 0$. Hence if $(a_j)_j \subset \ell^1_{v^{n}}(\La,c)$ is some sequence consisting of finitely supported sequences converging to $a \in \ell^1_{v^{n}}(\La,c)$, then
\begin{equation*}
	\| (a-a_j) \ast_c ((1+f(v)^{2})^{-1/2} b)\|_{\ell(\La)}^2 \leq \| a-a_j\|^2_{\ell^1_{v^n}(\La)} \|b\|^2_{\ell^2 (\La)} \to 0,
\end{equation*}
which shows that $a(1+D^2)^{-1/2}$ extends to a compact operator on $\ell^2 (\La)$. This shows that $(\ell_{v^n}^1 (\La,c), \ell^2 (\La)\oplus \ell^2 (\La), D)$ is a spectral triple for $C^* (\La,c)$ whenever $n \geq 1$. It is an even spectral triple since it is graded by
\begin{equation*}
	\gamma = \begin{pmatrix}
	1 & 0\\
	0 & -1
	\end{pmatrix}.
\end{equation*}
It remains to show that $(\ell_{v^n}^1 (\La,c), \ell^2 (\La)\oplus \ell^2 (\La), D)$ is a $QC^k$ spectral triple for $n \geq k+1$. Note that
\begin{equation*}
	|D| = \begin{pmatrix}
	 f(v) & 0 \\ 0& f(v)
	\end{pmatrix}
\end{equation*}
since $D$ is just a multiplication operator and $f(v) (\la) \geq  0$ for all $\la \in \La$. We also note that we can write out the commutator quite explicitly. An easy induction argument will show that
\begin{equation*}
	\ad^{k} (|D|)(a) = \sum_{i=0}^k (-1)^i \binom{k}{i} |D|^{k-i}a|D|^i =  \sum_{i=0}^k (-1)^i \binom{k}{i} |D|^{k-i}a|D|^i.
\end{equation*}
We need only look at what happens in one component.
Let $b \in \ell^2 (\La)$ and $a \in \ell^1_{v^{n}} (\La,c)$. Ignoring the fact that $D$ interchanges the two components, we have by slight abuse of notation
\begin{align*}
	\|\ad^k (|D|)(a) (b)\|_{\ell^2 (\La)}^2 &= \| \sum_{i=0}^k (-1)^i \binom{k}{i} |D|^{k-i}a|D|^i (b)\|_{\ell^2 (\La)}^2 \displaybreak[0]\\
	&= \sum_{ \la \in \La} \lvert \sum_{\mu \in \La} \sum_{i=0}^{k} (-1)^i \binom{k}{i} f(v) (\la)^{k-i} a(\mu)  f(v) (\la - \mu)^i b(\la-\mu) c(\mu,\la-  \mu)\rvert^2 \displaybreak[0]\\
	&\leq \sum_{ \la \in \La} \sum_{ \mu \in \La} \lvert \sum_{i=0}^{k} (-1)^i \binom{k}{i} f(v) (\la)^{k-i} a(\mu)f(v) (\la - \mu)^i b(\la-\mu) c(\mu,\la -\mu)\rvert^2 \displaybreak[0]\\
	&= \sum_{ \la \in \La} \sum_{ \mu \in \La} \lvert \sum_{i=0}^{k} (-1)^i \binom{k}{i} f(v) (\la + \mu)^{k-i} a(\mu)f(v) (\la)^i b(\la) c(\mu,\la)\rvert^2 \displaybreak[0]\\
	&= \sum_{ \la \in \La} \sum_{ \mu \in \La} |a(\mu )|^2 |b(\la)|^2 |c(\mu,\la)|^2 \lvert \sum_{i=0}^k (-1)^i \binom{k}{i} f(v) (\la + \mu)^{k-i} f(v) (\la)^i \rvert^2 \displaybreak[0]\\
	&=  \sum_{ \la \in \La} \sum_{ \mu \in \La} |a(\mu )|^2 |b(\la)|^2 \lvert ( f(v)(\la+\mu) - f(v) (\la) )^k \rvert^2 \displaybreak[0]\\
	&\leq \sum_{\la \in \La} \sum_{ \mu \in \La} | a(\mu)|^2 |b(\la)|^2 \Cdif^{2k} f(v) (\mu)^{2k} \displaybreak[0]\\
	&= \Cdif^{2k} \sum_{\mu \in \La} |a(\mu)|^2 |f(v) (\mu)|^{2k} \sum_{\la \in \La } |b(\la)|^2 \displaybreak[0]\\
	&= \Cdif^{2k} \|  a f(v)^k \|^2_{\ell^2 (\La)} \| b\|^2_{\ell^2 (\La)} \displaybreak[0]\\
	&\leq \Cdif^{2k} \|  a f(v)^{k} \|^2_{\ell^1 (\La)} \| b\|^2_{\ell^2 (\La)} \displaybreak[0]\\
	&= \Cdif^{2k} \|  a  \|^2_{\ell^1_{f(v)^{k}} (\La)} \| b\|^2_{\ell^2 (\La)} \displaybreak[0]\\
	&\leq \Cdif^{2k} \Cgr^{2k} \| a\|^2_{\ell^1_{v^{k}}(\La)} \| b\|^2_{\ell^2 (\La)},
\end{align*}
from which it follows that $a \in \Dom (\ad^k (|D|))$ if $a \in \ell^1_{v^{k}}(\La)$. In particular it holds for $a \in \ell^1_{v^n}(\La,c)$ as long as $n \geq k$. %Here we have used $|c(\mu,\la)| = 1$ for all $\la , \mu \in \La$, $|f(v)(\la+\mu)| - |f(v) (\la)| \leq C'_{v,p,r} |f(v) (\mu)|$ and $\|  a  \|_{\ell^1_{|p(v)^{kr}|} (\La)} \leq C^{''}_{v,p,r} \| a\|_{\ell^1_{v^{kr\deg p}}(\La)}$ for some constants $C'_{v,p,r}$ and $C^{''}_{v,p,r}$. The latter two facts follow by \cref{lemma:polynomial-approx}. 
Along the same lines we verify that $[D,a] \in \Dom (\ad^k (|D|))$ for $a \in \ell^1_{v^n}(\La,c)$, $n \geq k+1$. Ignoring the fact that $D$ interchanges the two components, we have for $b \in \ell^2 (\La)$ by slight abuse of notation
\begin{align*}
&\| \ad^k (|D|)([D,a])(b)\|^2_{\ell^2 (\La)} = \| \sum_{i=0}^k (-1)^i  \binom{k}{i} |D|^{k-i} [D,a]|D|^{i} b \|^2_{\ell^2 (\La)} 
\\
&= \sum_{\la \in \La} \lvert \sum_{ \mu \in \La} \sum_{i=0}^{k} (-1)^i \binom{k}{i} \bigg( f(v) (\la)^{k+1-i} a(\mu) f(v)(\la-\mu)^i b(\la - \mu) c(\mu,\la - \mu) \\
&\qquad -f(v) (\la)^{k-i} a(\mu) f(v) (\la-\mu)^{i+1} b(\la - \mu) c(\mu,\la-\mu) \bigg)\rvert^2 \displaybreak[0]
\\
&\leq  \sum_{\la \in \La}  \sum_{ \mu \in \La} \lvert \sum_{i=0}^{k} (-1)^i \binom{k}{i} \bigg( f(v) (\la)^{k+1-i} a(\mu) f(v)(\la-\mu)^i b(\la - \mu) c(\mu,\la - \mu) 
\\
&\qquad -f(v) (\la)^{k-i} a(\mu) f(v) (\la-\mu)^{i+1} b(\la - \mu) c(\mu,\la-\mu) \bigg)\rvert^2 \displaybreak[0]
\\
&= \sum_{\la \in \La}  \sum_{ \mu \in \La} \lvert \sum_{i=0}^{k} (-1)^i \binom{k}{i} \bigg( f(v) (\la + \mu)^{k+1-i} a(\mu) f(v)(\la)^i b(\la) c(\mu,\la)
\\
&\qquad -f(v) (\la+\mu)^{k-i} a(\mu) f(v) (\la)^{i+1} b(\la) c(\mu,\la) \bigg)\rvert^2 \displaybreak[0]
\\
&= \sum_{\la \in \La}  \sum_{ \mu \in \La} |a(\mu)|^2 |b(\la)|^2 |c(\mu,\la)|^2 \lvert \sum_{i=0}^{k} (-1)^i \binom{k}{i} \bigg( f(v) (\la + \mu)^{k+1-i} f(v)(\la)^i
\\
&\qquad -f(v) (\la+\mu)^{k-i}  f(v) (\la)^{i+1}   \bigg)\rvert^2 \displaybreak[0]
\\
&= \sum_{\la \in \La}  \sum_{ \mu \in \La} |a(\mu)|^2 |b(\la)|^2  |f(v) (\la+\mu) - f(v) (\la)|^2 |(f(v) (\la+\mu) - f(v) (\la))^k|^2 \displaybreak[0]
\\
&\leq \sum_{\la \in \La}  \sum_{ \mu \in \La} |a(\mu)|^2 |b(\la)|^2 \Cdif^2 f(v) (\mu)^2 \Cdif^{2k} f(v) (\mu)^{2k} \displaybreak[0]
\\
&= \Cdif^{2k+2} \sum_{\mu \in \La} |a(\mu)|^2 f(v) (\mu)^{2(k+1)} \sum_{ \la \in \La} | b(\la)|^2 \displaybreak[0]\\
&= \Cdif^{2k+2} \| a f(v)^{k+1} \|^2_{\ell^2 (\La)} \| b\|^2_{\ell^2 (\La)} \displaybreak[0]
\\
&\leq \Cdif^{2k+2} \| a f(v)^{k+1} \|^2_{\ell^1 (\La)} \| b\|^2_{\ell^2 (\La)} \displaybreak[0]
\\
&= \Cdif^{2k+2} \| a \|^2_{\ell^1_{f(v)^{k+1}}(\La)}  \| b\|^2_{\ell^2 (\La)} \\
&\leq \Cgr^{2k+2} \Cdif^{2k+2} \| a \|^2_{\ell^1_{v^{k+1}}(\La)}  \| b\|^2_{\ell^2 (\La)}.
\end{align*}
We then see that $\ell^1_{v^{n}}(\La,c) \subset QC^*(\La,c)_k$ if $n \geq k+1$, which finishes the proof.
\end{proof}
Now let $D$ be given by
\begin{equation}
	\label{eq:def-D-2}
	D = \begin{pmatrix}
	f(v) & 0 \\
	0 & f(v)
	\end{pmatrix}.
\end{equation}
Then $D = |D|$, and
the following is also true by more or less the same proof as above except for the grading.
\begin{prop}
	\label{prop:n-geq-k+1-gives-QCk-2}
	Let $v$ be a weight on $\tfp{G}$, let $\La$ be a lattice in $\tfp{G}$, let $f$ be a spectral triple compatible function for $v$ with respect to $\La$, and let $D$ be defined by \eqref{eq:def-D-2}. Then
	$(\ell_{v^n}^1 (\La,c), \ell^2 (\La)\oplus \ell^2 (\La), D)$ is a spectral triple for $C^*(\La,c)$ whenever $n\geq 1$, and $\ell_{v^n}^1 (\La,c) \subset QC^* (\La,c)_k$ for $n\geq k+1$. In other words, if $n\geq k+1$ then the spectral triple is quantum $C^k$. 
\end{prop}

\subsection{Modulation spaces as smooth modules}
\label{section:modulation-spaces-as-smooth-modules}
In \cref{prop:n-geq-k+1-gives-QCk} we saw how to obtain an even $QC^k$ spectral triple $(\ell_{v^n}^1 (\La,c), \ell^2 (\La)\oplus \ell^2 (\La), D)$ for $C^* (\La,c)$ whenever $n \geq k+1$. The goal of this section is to show how the Heisenberg module $\heis{G}{\La}$ of \cref{section:modulation-spaces-as-modules} can be equipped with a $QC^k$-structure for any $k \in \N$. The proof follows the lines of \cite[Proposition 2.1]{Ri81} and \cite[Proposition 3.7]{Ri88}.
\begin{prop}
	\label{prop:smooth-struct-on-heis-mod}
	Let $v$ be a weight on $\tfp{G}$, let $\La$ be a lattice in $\tfp{G}$, let $f$ be a spectral triple compatible function for $v$ with respect to $\La$, and let $C^*(\La,c)$ be given a $QC^k$-structure by ways of \cref{prop:n-geq-k+1-gives-QCk} or \cref{prop:n-geq-k+1-gives-QCk-2} for some $k \in \N$. Then there is a uniformly norm bounded approximate unit $(e_m)_{m=1}^{\infty}$ of the form \eqref{eq:approx-id-form} such that $(\heis{G}{\La}, (e_m)_{m=1}^{\infty})$  is a $QC^k$-module over $C^* (\La,c)$. 
\end{prop}
\begin{proof}
	We first fix $k \in \N$. It suffices to prove that we can find a (uniformly norm bounded) approximate unit $(e_m)_{m=1}^{\infty}$ where
	\begin{equation*}
		e_m = \sum_{i=1}^m \modft_{g_i , g_i},
	\end{equation*}
	for which $\lhs{g_i}{g_j} \in \ell^1_{v^{k+1}} (\La, c)\subset QC^*(\La,c)_k$ for all $i,j \in \{1,\ldots, m\}$ and all $m \in \N$, as determined by \cref{prop:n-geq-k+1-gives-QCk} or \cref{prop:n-geq-k+1-gives-QCk-2}. Indeed, we will find a unit. Note first that $\heis{G}{\La}$ is a $C^* (\La,c)$-$C^* (\Lao, \overline{c})$-equivalence bimodule and both $C^*$-algebras are unital. Moreover, we know that $M^1_{v^{k+1}}(G)$ is an $\ell^1_{v^{k+1}}(\La,c)$-$\ell^1_{v^{k+1}}(\Lao,\overline{c})$-pre-equivalence bimodule. Now note that $\ell^1_{v^{k+1}}(\Lao,\overline{c})$ is unital with the same unit as $C^* (\Lao, \overline{c})$. Furthermore, $\ell^1_{v^{k+1}}(\Lao,\overline{c})$ is spectral invariant in $C^* (\Lao, \overline{c})$ by \cref{prop:GrLe-spec-inv}.
	Hence we are in the situation of \cref{prop:results-from-aujalu}.
	Since $M^1_{v^{k+1}}(G)$ is a pre-equivalence bimodule, we may find finitely many elements $h_1, \ldots ,h_l, h'_1 , \ldots , h'_l \in M^1_{v^{k+1}}(G)$ such that $\sum_{i=1}^l \rhs{h_i}{h'_i}$ is invertible. As $\ell^1_{v^{k+1}}(\Lao,\overline{c})$ is spectral invariant in $C^* (\Lao, \overline{c})$, it follows that
	\begin{equation*}
		\bigg( \sum_{i=1}^l \rhs{h_i}{h'_i} \bigg)^{-1} \in \ell^1_{v^{k+1}}(\Lao,\overline{c}).
	\end{equation*}
	If we then set $h''_i = h'_i \cdot ( \sum_{i=1}^l \rhs{h_i}{h'_i} )^{-1} \in M^1_{v^{k+1}}(G)$, we get
	\begin{equation*}
		\sum_{i=1}^l \rhs{h_i}{h''_i} = 1_{\ell^1_{v^{k+1}}(\Lao,\overline{c})} = 1_{C^* (\Lao, \overline{c})} = 1_{C^* (\Lao, \overline{c})}^* = \bigg( \sum_{i=1}^l \rhs{h_i}{h''_i}\bigg)^* = \sum_{i=1}^l \rhs{h''_i}{h_i} .
	\end{equation*}
	But then $(h_i)_{i=1}^l$ is a module frame for $\heis{G}{\La}$ by \cref{prop:dual-seq-implies-frame} with $h_i \in M^1_{v^{k+1}}(G)$ for all $i=1,\ldots ,l$. It then follows by \cref{prop:results-from-aujalu} that there is $g_i \in M^1_{v^{k+1}}(G)$ such that $\sum_{i=1}^l \rhs{g_i}{g_i} = 1_{C^* (\Lao, \overline{c})}$. For any $f \in \heis{G}{\La}$ we then have
	\begin{equation*}
		\sum_{i=1}^l \modft_{g_i, g_i} f = \sum_{i=1}^l \lhs{f}{g_i} g_i = \sum_{i=1}^l f \rhs{g_i}{g_i} = f \sum_{i=1}^{l} \rhs{g_i}{g_i} = f 1_{C^* (\Lao, \overline{c})} = f,
	\end{equation*}
	which shows that $(g_i)_{i=1}^l$ has the desired property. Since $\lhs{g_i}{g_j}\in \ell^1_{v^{k+1}} (\La,c)$ for all $i,j=1,\ldots,l$, it follows that $(\heis{G}{\La}, (g_i)_{i=1}^l)$ is a $QC^k$-module over $C^* (\La,c)$. 
\end{proof}

Remark that even in the case of elementary groups as in \cite{Ri88} the above results are stronger than just being able to find tight module frames with elements in the Schwartz space. Indeed, in case $G = \R^d$, $d \in \N$, and $\La$ is a lattice in $\R^d \times \widehat{\R^d}\cong \R^{2d}$, Parseval module frames with elements in Schwartz space $\mathcal{S}(\R^d)$ would give the Heisenberg module a $QC^{\infty}$-structure. However, the Feichtinger algebra approach gives the possibility of finding Parseval module frames which give the Heisenberg module a $QC^k$-structure, which is not simultaneously a $QC^{k+1}$-structure. We give some examples for the noncommutative $2$-torus in \cref{Sec:NC2T}.

\subsection{The link to Gabor analysis}
\label{sec:Link-Gab-an}
The existence of sufficiently regular approximate identities from \cref{def:QCk-structure-mods} is in the setting of Heisenberg modules a result about existence of multiwindow Gabor frames with windows in suitably weighted Feichtinger algebras.

The following result is a special case of \cite[Theorem 3.11]{AuEn19}.
\begin{prop}
	\label{thm:generators_multiwindow}
	Let $G$ be a second countable LCA group, let $\La \subset \tfp{G}$ be a lattice, and let $g_1, \ldots, g_l$ be elements of the Heisenberg module $\heis{G}{\La}$. Then the following are equivalent:
	\begin{enumerate}
		\item The set $\{ g_1, \ldots, g_l \}$ is a Parseval module frame for $\heis{G}{\La}$ as a left $C^*(\La,c)$-module. That is, for all $f \in \heis{G}{\La}$ we have
		\[ f = \sum_{j=1}^l \lhs{f}{g_j}  g_j = \sum_{j=1}^l f\rhs{g_j}{g_j} .\]
		\item The system
		\[ \mathcal{G}(g_1, \ldots, g_l ; \La) = \{ \pi(\la) g_j : \la \in \La, 1 \leq j \leq l \} \]
		is a Parseval multiwindow Gabor frame for $L^2(G)$.
	\end{enumerate}
\end{prop}
The following is then immediate by \cref{section:modulation-spaces-as-smooth-modules} and \cref{thm:generators_multiwindow}.
\begin{thm}
Let $v$ be a weight on $\tfp{G}$, let $\La$ be a lattice in $\tfp{G}$, let $f$ be a spectral triple compatible function for $v$ with respect to $\La$, and let $C^*(\La,c)$ be given a $QC^k$-structure by ways of \cref{prop:n-geq-k+1-gives-QCk} or \cref{prop:n-geq-k+1-gives-QCk-2} for some $k \in \N$. Then a Parseval multiwindow Gabor frame $\GS (g_1, \ldots , g_l ; \La)$ for $L^2 (G)$ with $g_j \in M^1_{v^{n}}(G)$, $j=1,\ldots,l$, $n\geq k+1$, gives the Heisenberg module $\heis{G}{\La}$ the structure of a $QC^k$-module over $C^* (\La,c)$.
	%Let $\La \subset \tfp{G}$ be a lattice in a second countable LCA group $G$, let $v$ be a commutator bounded weight on $\La$, and let $f: \La \to \R$ be spectral triple compatible with respect to $v$. Moreover, let $n \geq k+1$, so that 
	%\begin{equation*}
	%	(\ell^1_{v^{n}} (\La,c) , \ell^2 (\La) \oplus \ell^2 (\La), D )
	%\end{equation*}
	%gives a $QC^k$-structure on $C^* (\La,c)$, where $D$ is defined by either \eqref{eq:def-D-spectral-op} or \eqref{eq:def-D-2}. Then a Parseval multiwindow Gabor frame $\GS (g_1, \ldots , g_l ; \La)$ for $L^2 (G)$ with $g_j \in \ell^1_{v^{n}}(G)$ gives the Heisenberg module $\heis{G}{\La}$ the structure of a $QC^k$-module over $C^* (\La,c)$.
\end{thm}

\subsection{Example: The noncommutative $2$-torus}
\label{Sec:NC2T}
	We refer the reader to \cite{caphre11} or \cite{Va06} for details on this example. What follows will also build on \cref{ex:Heis-mod-for-R2}.
	
	On the noncommutative $2$-torus, denoted $C^* (\La,c)$ in this section, there are two canonical unbounded derivations denoted by $\partial_1$ and $\partial_2$. They are given by
	\begin{equation*}
		\begin{split}
		\partial_1 &: (a(x,\omega))_{(x,\omega)\in \La} \mapsto (2\pi i x a(x,\omega))_{(x,\omega)\in \La} \\
		\partial_2 &: (a(x,\omega))_{(x,\omega)\in \La} \mapsto (2\pi i \omega a(x,\omega))_{(x,\omega)\in \La},
		\end{split}
	\end{equation*}
	for $(a(x,\omega))_{(x,\omega)\in \La} \in C^* (\La,c)$.
	These are only densely defined, but we see that $\ell^1_v(\La,c)\subset \Dom \partial_i$ for $i=1,2$, where $v$ is the weight $v(x,\omega) = (1+x^2 + \omega^2)^{1/2}$. %Then $v$ is a commutator bounded weight. 
	In the rest of this section $v$ will denote this weight. We may then consider the triple for the noncommutative $2$-torus given by
	\begin{equation*}
		(\ell^1_v (\La,c) , \ell^2 (\La) \oplus \ell^2 (\La), D )
	\end{equation*}
	where $D$ is the unbounded operator given by
	\begin{equation}
		D = \begin{pmatrix}
		0 & \partial_1 + i \partial_2 \\
		- \partial_1 + i\partial_2 & 0
		\end{pmatrix}.
	\end{equation}

	\begin{lemma}
	\label{lemma:canonical-spec-trip}
	The triple 
	\begin{equation*}
		(\ell^1_v (\La,c) , \ell^2 (\La) \oplus \ell^2 (\La), D )
	\end{equation*}
	defined above is a spectral triple for $C^* (\La,c)$.
	\end{lemma}
	\begin{proof}
	    For $a \in \ell^1_v (\La,c)$ it follows by the Leibniz rule for $\partial_i$, $i=1,2$, that $a \cdot \Dom (D) \subset \Dom (D)$. Moreover, a standard calculation will show that the commutator $[D,a]$ extends to left multiplication by the matrix
	    \begin{equation}
	        \begin{pmatrix}
	        0 & \partial_1 (a) + i \partial_2 (a) \\
	        - \partial_1 (a) + i\partial_2 (a) & 0
	        \end{pmatrix},
	    \end{equation}
	    which is a bounded operator. That $a(1+D^2)^{-1/2}$ extends to a compact operator follows as in the proof of \cref{prop:n-geq-k+1-gives-QCk}.
	\end{proof}
	\begin{rmk}
	    The spectral triple of \cref{lemma:canonical-spec-trip} is also known as the \emph{canonical spectral triple for the noncommutative $2$-torus}. However, the $*$-subalgebra of $C^* (\La,c)$ typically chosen is the one consisting of the Schwartz sequences. 
	\end{rmk}
	
	$D$ is a selfadjoint operator and $D^2$ is the multiplication operator given by
	\begin{equation*}
		D^2 = 
		\begin{pmatrix}
		4\pi^2 (x^2 + \omega^2) & 0 \\
		0 & 4\pi^2 (x^2 + \omega^2)
		\end{pmatrix}.
	\end{equation*}
	We now let $f(v) = 2\pi (v^2 -1)^{1/2}$. Then $f$ is spectral triple compatible for $v$ with respect to any lattice $\La \subset \R^2$, and we obtain
	\begin{equation*}
		\begin{pmatrix}
		f(v) & 0 \\
		0 & f(v) 
		\end{pmatrix}
		= 
		\begin{pmatrix}
		2\pi (x^2 + \omega^2)^\frac{1}{2} & 0 \\
		0 & 2\pi (x^2 + \omega^2)^\frac{1}{2} 
		\end{pmatrix}
		= |D|.
	\end{equation*}
	Hence 
	\begin{equation*}
	(\ell^1_v (\La,c) , \ell^2 (\La) \oplus \ell^2 (\La), |D| )
	\end{equation*}
	which we create by ways of \cref{section:smooth-structures-algebra} is related to the canonical spectral triple for the noncommutative $2$-torus. By \cref{prop:n-geq-k+1-gives-QCk-2}, $$(\ell^1_{v^{n}} (\La,c) , \ell^2 (\La) \oplus \ell^2 (\La), |D| )$$ equips the noncommutative $2$-torus with a $QC^k$-structure if $n \geq k+1$. However $$(\ell^1_{v^{n}} (\La,c) , \ell^2 (\La) \oplus \ell^2 (\La), D )$$ also equips the noncommutative $2$-torus with a $QC^k$-structure if $n \geq k+1$. We saw that it defined a spectral triple in \cref{lemma:canonical-spec-trip}. That $a \in \ell^1_{v^{n}} (\La,c)$ is such that $a \in \Dom(\ad^k (|D|)$ for $k \geq n+1$ follows exactly as in the proof of \cref{prop:n-geq-k+1-gives-QCk}. If we realize that $\partial_1 (a) + i \partial_2 (a) \in \ell^1_{v^{n-1}} (\La,c)$ and  $- \partial_1 (a) + i\partial_2 (a) \in \ell^1_{v^{n-1}} (\La,c)$, it also follows that $[D,a] \in \Dom (\ad^k (|D|))$ for $n \geq k+1$ by essentially the same argument as in the proof of \cref{prop:n-geq-k+1-gives-QCk}, since we did that proof looking only at one component. Hence the twisted convolution algebra $\ell^1_{v^{n}} (\La,c)$ becomes a suitable $*$-subalgebra to give the noncommutative $2$-torus a $QC^k$-structure for $n \geq k+1$ both for the canonical spectral triple and for the spectral triple constructed by ways of \cref{section:smooth-structures-algebra}.

By \cref{sec:Link-Gab-an} we may then equip Heisenberg modules with $QC^k$-structures by finding suitably regular multiwindow Gabor frames. We illustrate this with some examples. Note however that there are very few functions $g \in L^2 (\R)$ for which the set 
\begin{equation*}
    \{ \La \subset \tfp{\R} \mid \text{ $\La$ is a lattice and $\GS (g; \La)$ is a frame for $L^2 (\R)$}\}
\end{equation*}
is known.
\begin{exmp}[$QC^\infty$-structures]
Let $\La = \alpha \Z \times \beta \Z$ be a lattice in $\R^2$ with $\alpha,\beta >0$ and $\alpha \beta <1$. This yields a Heisenberg module $\heis{\R}{\La}$ by the constructions above. A celebrated result in time-frequency analysis tells us that time-frequency shifts of the Gaussian $g (t) = 2^{1/4} e^{-\pi t^{2}}$ determines a Gabor frame $\GS(g, \alpha \Z \times \beta \Z)$ for $L^2 (\R)$ if and only if $\alpha \beta <1$, see \cite{ly92}, \cite{se92-1}. If $S$ then denotes the frame operator with respect to $g$, $\GS (S^{-1/2} g, \alpha \Z \times \beta \Z)$ is a Parseval frame for $L^2 (\R)$. By \cref{prop:results-from-aujalu} it follows that if $g \in M^1_{v^{s}} (\R)$, $s\in [0,\infty)$, so is $S^{-1/2} g$. But $g$ is a Schwartz function, hence it is in $\cap_{s\geq 0} M^1_{v^s} (\R) = \mathcal{S}(\R)$ \cite[Proposition 11.3.1]{gr01}, where $\mathcal{S}(\R)$ denotes the Schwartz functions on $\R$. It follows that $\{ S^{-1/2} g \}$ gives the Heisenberg module $\heis{\R}{\La}$ a $QC^{\infty}$-structure for all $\alpha \beta <1$.
\end{exmp}
\begin{exmp}[$QC^k$-structure]
Let $g$ be a function in $M^1_{v^{k+1}}(\R)$. Then by \cite[p. 120]{gr01} $\GS (g;\La)$ is a frame for $L^2 (\R)$ for some $\La = \alpha \Z \times \beta \Z$, as long as $\alpha, \beta>0$ are small enough. Let $S$ be the frame operator of $g$. Then as above it follows that $S^{-1/2} g \in M^1_{v^{k+1}}(\R)$ also. As in the previous example it follows that $\{S^{-1/2}g\}$ then gives $\heis{\R}{\La}$ a $QC^k$-structure.

For explicit examples of $QC^k$-structures on Heisenberg modules that are not simultaneously $QC^{\infty}$-structures one may use B-splines $B_N$, see \cite[Section A.8, Section 11.7]{Chr16}. It is known that $\GS(B_N, \La)$ is a frame for $L^2 (\R)$ whenever $\La = \alpha \Z \times \beta \Z \subset \R^2$ is such that $\alpha \in (0,N)$ and $\beta \in (0,1/N]$ \cite[Corollary 11.7.1]{Chr16}. Values of $k$ for which a given $B_N$ gives a Heisenberg module a $QC^k$-structure can be done via the Rihaczek distribution $R(g,g)(x,\omega)=g(x)\overline{\widehat{g}}(\omega)e^{-2\pi ix\omega}$.
\end{exmp}

For the following example, note that if $g \in L^2 (\R)$ and $\La \subset \tfp{\R}$ is so that $\GS(g;\La)$ is a frame for $L^2 (\R)$, then $s(\La) \leq 1$ \cite{rast95}. %, where $vol (\La)$ is the volume of the lattice $\La$ \cite{rast95}. 
For $\La = \alpha \Z \times \beta \Z$, $\alpha,\beta >0$, $s(\La) = \alpha\beta$. However, even for $s(\La) > 1$ we may construct Heisenberg modules. To obtain $QC^k$-structures on such Heisenberg modules $\heis{\R}{\La}$ we need several generators. %Note that \cite[Theorem 5.1]{JaLu18} is valid in our situation, even with weights, since all that is needed is spectral invariance of the Banach $*$-subalgebras $M^1_v (\La)$ in $C^* (\La,c)$. 
\begin{exmp}[Multiple generators] Suppose $\GS(g;\La)$ is a Gabor system for $L^2 (\R)$ and $s(\Lambda)\in[l-1,l)$. Then there exist points $z_1,...,z_l$ in $\RR^2$ and a lattice $\Lambda_0$ such that $\Lambda=z_1\Lambda_0\cup\cdots\cup z_n\Lambda_0$ with $s(\Lambda_0)<1$, see  the proof of \cite[Corollary 5.6]{JaLu18}. Hence if $\GS(g,\Lambda_0)$ is a Gabor frame, then $\GS(\pi(z_1)g,...,\pi(z_l)g;\La)$ is a multi-window Gabor frame for $L^2(\R)$.  

In particular, let $g$ be the Gaussian and let $\alpha\beta$ be in $[l-1,l)$ for some $n\in\N$. Then there exist $z_1,...,z_l$ in $\RR^2$ such that $\GS(\pi(z_1)g,...,\pi(z_l)g;\alpha\ZZ\times\beta\ZZ)$ is a multiwindow Gabor frame for $L^2(\R)$. Hence if $S$ is the multiwindow frame operator for $\GS(\pi(z_1)g,...,\pi(z_l)g;\alpha\ZZ\times\beta\ZZ)$, then  $\{ S^{-1/2}g_1, \ldots,S^{-1/2}g_n \}$ implements a $QC^\infty$-structure on $\heis{\R}{\La}$. 
\end{exmp}

\subsection{Example: The noncommutative solenoid}

Noncommutative solenoids have attracted some interest in the theory of operator algebras \cite{LaPa13,LaPa15} and time-frequency analysis. We follow the presentation in  \cite{EnJaLu19}, where Heisenberg modules over noncommutative solenoids have been linked with Gabor frames for lattices in $\R\times\Qp$. 

Given a prime number $p$, the \emph{$p$-adic absolute value} on $\Q$ is defined by $|x|_p = p^{-k}$, where $x = p^k(a/b)$ and $p$ divides neither $a$ nor $b$. For $x=0$ we set $|0|_p = 0$. The $p$-adic absolute value satisfies the ultrametric triangle inequality, that is,
\begin{equation}
| x + y |_p \leq \max \{ |x|_p, |y|_p \}. \label{eq:ultrametric}
\end{equation}
The completion of $\Q$ with respect to the metric $d_p(x,y) = |x-y|_p$ is a field denoted by $\Q_p$ and its elements are called \emph{$p$-adic numbers}. The topology inherited from the metric makes $\Q_p$ into a locally compact Hausdorff space. Moreover, $\Qp$ is a second countable locally compact abelian group with respect to the topology induced by the above metric and under addition. One can show that every $p$-adic number $x$ has a \emph{$p$-adic expansion} of the form
\[ x = \sum_{k=-\infty}^\infty a_k p^k, \]
where $a_k \in \{ 0, \ldots, p-1 \}$ for each $k$ and there exists some $n \in \Z$ such that $a_k =0$ for all $k < n$. The sequence $(a_k)_{k 
\in \Z}$ in this expansion is unique.

The closed unit ball in $\Q_p$ is denoted by $\Z_p$ and its elements are called \emph{$p$-adic integers}. Because of (\ref{eq:ultrametric}) and the multiplicativity of $| \cdot |_p$, $\Z_p$ is a subring of $\Q_p$. In terms of $p$-adic expansions, a $p$-adic number $x = \sum_{k \in \Z} a_k p^k$ is a $p$-adic integer if and only if $a_k = 0$ for $k < 0$. The map $ \{ 0, \ldots, p-1 \}^{\N} \to \Z_p$ given by $(a_k)_k \mapsto \sum_k a_k p^k$ is a homeomorphism, which shows that $\Z_p$ has the topology of a Cantor set. In particular, $\Z_p$ is a compact subgroup of $\Q_p$. But $\Z_p$ is also open in $\Q_p$. Indeed, if $x \in \Z_p$, then using (\ref{eq:ultrametric}) one shows that the open ball $B_{1/2}(x) = \{ y \in \Q_p : |y-x|_p < 1/2 \}$ is contained in $\Z_p$.

We take the Haar measure $\mu_{\Qp}$ on $\Q_p$ so that $\mu_{\Qp}(\Zp)=1$. The Haar measure on $\Z_p$ is the one on $\Q_p$ restricted to $\Z_p$.

We denote by $\Z[1/p]$ the subring of $\Q$ consisting of rational numbers of the form $a/p^k$ where $k,a\in \Z$. Then $\Q_p = \Z_p + \Z[1/p]$ and $\Z_p \cap \Z[1/p] = \Z$, so that
\[ \Q_p / \Z_p = \frac{\Z_p + \Z[1/p]}{ \Z_p} \cong \frac{ \Z[1/p] }{ \Z_p \cap \Z[1/p] } = \Z[1/p] / \Z \]
as abelian groups. Denote the quotient map $\Q_p \to \Z[1/p]/\Z$ by $x \mapsto \{ x \}_p$. In terms of $p$-adic expansions, we have $\left\{ \sum_{k \in \Z} a_k p^k \right\}_p = \sum_{k=-\infty}^{-1} a_k p^k + \Z$ (observe that for any $p$-adic number $x$ only finitely many of the $a_{k}$ are nonzero). Every character $\omega \in \widehat{\Q}_p$ is of the form
\[ \omega: \Qp \to \C, \ \omega(x) = e^{2\pi i \{ xy \}_p }, \ x \in \Qp, \]
for some $y \in \Q_p$. In fact, the map $\Q_p \to \widehat{\Q}_p$ given by mapping $y$ to the $\omega$ defined above is a topological isomorphism. Hence the Pontryagin dual of $\Q_p$, $\widehat{\Q}_{p}$, can be identified with $\Q_p$ itself.

Every $y=(y_{\infty},y_{p})\in \R\times\Qp$ defines a character $\omega_{y}\in \widehat{\R}\times\widehat{\Q}_{p}$ via
\begin{equation} \label{eq:2006b} \omega_{y} : \R\times\Qp, \ x=\big(x_{\infty},x_{p}\big) \mapsto e^{2\pi i (x_{\infty} y_{\infty}-\{ x_{p} y_{p}\}_{p})}.
\end{equation}
One can show that every character in $\widehat{\R}\times\widehat{\Q}_{p}$ is given as in \eqref{eq:2006b} for some $y\in \R\times\Qp$.

There is an abundance of lattices in $\R\times\Qp$. This is well-known and can be found in, e.g., \cite{LaPa15}.
\begin{prop}
Let $p$ be a prime number. For any $\alpha \in \R \setminus \{ 0 \}$ the mapping
\[ \psi_{\alpha}: \Z[1/p]\to\R\times\Qp, \, \psi_{\alpha}(q) = (\alpha q, q )\]
embeds $\Z[1/p]$ as a lattice into $\R\times\Qp$. The set
$B_{\alpha} = [0,|\alpha|) \times \Z_p$
is a fundamental domain for $\psi_{\alpha}(\Z[1/p])$ in $\R\times\Qp$ and $s(\psi_{\alpha}(\Z[1/p])=\vert \alpha\vert$. 
Moreover, under the identification of $\R\times\Qp$ with $\widehat{\R}\times\widehat{\Q}_{p}$ as in \eqref{eq:2006b}, the group $\psi_{\alpha}(\Z[1/p])^{\perp}$ can be identified with $\psi_{1/\alpha}(\Z[1/p])$.
\end{prop}

For the construction of smooth structures we consider the weighted Feichtinger algebras 
$M^1_{v_s}(\R\times\Qp)$, where the weight $v_s(x,\omega,q,r)=(1+|x|^2+|\omega|^2+|q|^2+|r|^2)^{s/2}$ for $s\ge 0$. It follows from, e.g., \cite[Theorem 7]{Fe81}, that the functions in $M^1_{v_s}(\R\times\Qp)$ are exactly those of the form
\begin{equation} \label{eq:1806b} f = \sum_{j\in\N} f^{(\R)}_{j} \otimes f^{(\Qp)}_{j} \ \ \text{where} \ \ f^{(\R)}_{j}\in M^1_{v_s}(\R),\, f^{(\Qp)}_{j}\in M^1_{v_s}(\Qp)\end{equation}
for all $j\in\N$ and such that $\sum_{j\in\N} \Vert f^{(\R)}_{j} \Vert_{M^1_{v_s}(\R)} \, \Vert  f^{(\Qp)}_{j} \Vert_{M^1_{v_s}(\Qp)} <\infty$.
The norm on $M^1_{v_s}(\R\times\Qp)$ is given by
\[ \Vert f \Vert_{M^1_{v_s}(\R\times\Qp)} = \inf \big\{ \sum_{j\in\N} \Vert f^{(\R)}_{j} \Vert_{M^1_{v_s}(\R)} \, \Vert  f^{(\Qp)}_{j} \Vert_{M^1_{v_s}(\Qp)} \big\},\]
where the functions $f$, $\{f_{j}^{(\R)}\}_{j\in\N}$ and $\{ f_{j}^{(\Qp)}\}_{j\in\N}$ are related as in \eqref{eq:1806b} and the infimum is taken over all possible representations of $f$ as in \eqref{eq:1806b}. A function space constructed with the help of weighted Feichtinger algebras for the p-adics is also investigated in \cite{DiJa19}.

For every $\omega = (\omega_{\infty},\omega_{p})\in \R\times\Qp$ we define the modulation operator by
\[ M_{\omega} f(t_{\infty},t_{p}):= M_{\omega_{\infty},\omega_{p}} f(t_{\infty},t_{p}) = e^{2\pi i (\omega_{\infty} t_{\infty}-\{ \omega_{p} t_{p}\}_{p})}f(t_{\infty},t_{p}), \ (t_{\infty},t_{p}) \in \R\times\Qp.\]
A Gabor system generated by a function $g\in L^{2}(\R\times\Qp)$ and the lattice 
\[ \La = \psi_{\alpha}(\Z[1/p])\times \psi_{\beta}(\Z[1/p]) = \left\{  (\alpha q, q, \beta r, r) : q,r\in \Z[1/p] \right\}, \ \alpha,\beta>0\] 
is thus of the form
\[ \{ \pi(\lambda) g\}_{\la\in\La} = \big\{ (t_{\infty},t_{p}) \mapsto e^{2\pi i (\beta r t_{\infty}-\{ r t_{p}\}_{p})} g(t_{\infty}-\alpha q,t_{p}-q) \big\}_{q,r\in\Z[1/p]}. \]

We introduce a noncommutative solenoid as the twisted group $C^*$-algebra $C^*(\Lambda, c)$ of $\Lambda$, see \cite{LaPa13,LaPa15}. 
%For $\alpha,\beta>0$ we define the following two Banach algebras:
%\begin{align*} \lmodule & = \big\{ \mathbf{a} \in \mathbb{B}(L^{2}(\R\times\Qp)) \, : \, \mathbf{a} = \sum_{q,r\in\Z[1/p]} a(q,r) E_{\beta %r,r}T_{\alpha q,q}, \ a\in \ell^{1}_{s}(\Z[1/p]^{2}) \big\}, \\
%\rmodule & = \big\{ \mathbf{b} \in \mathbb{B}(L^{2}(\R\times\Qp)) \, : \, \mathbf{b} = \frac{1}{\alpha\beta}\sum_{q,r\in\Z[1/p]} b(q,r) \big( %E_{\alpha^{-1} r,r}T_{\beta^{-1} q,q}\big)^{*}, \ b\in \ell^{1}_{s}(\Z[1/p]^{2}) \big\}.\end{align*}\todo{In the rest of the paper we have used the corresponding sequence spaces. I don't think we should change this now.}

%Indeed, the norms $\Vert \mathbf{a} \Vert_{\mathcal{A}} = \Vert a \Vert_{\ell^1_s}$, $\Vert \mathbf{b}\Vert_{\rmodule} = \Vert b \Vert_{\ell^1_s}$ (where $\mathbf{a},a,\mathbf{b}$ and $b$ are related as above) turn $\lmodule$ and $\rmodule$ into involutive Banach algebras with respect to composition of operators and where the involution is the $L^{2}$-adjoint. 

Observe that $C^*(\Lambda, c)$ is \emph{not} generated by finitely many unitaries as is the case of the noncommutative 2-torus.

We specialize the definition of a Dirac operator \eqref{eq:def-D-spectral-op} to $\RR\times\Qp$ for the lattice $\Lambda=\{ (\alpha q, q, \beta r, r) : q,r \in \Z[1/p]\}$ 
\begin{equation}
	D= \begin{pmatrix}
	0 & (1+|x|^2+|\omega|^2+|q|^2+|r|^2)^{s/2} \\ 
	(1+|x|^2+|\omega|^2+|q|^2+|r|^2)^{s/2} & 0
	\end{pmatrix}.
\end{equation}
Hence we have constructued a spectral triple on noncommutative solenoids, which as far as we know has not been considered before in the literature. One of the results in \cite[Corollary 3.3]{EnJaLu19} allows us to construct $QC^k$ structures on the Heisenberg module $E_{\R\times\Qp,\Lambda}$. 
\begin{prop}
For any $g^{(\R)}\in\SO(\R)$ and $\alpha,\beta>0$ the following statements are equivalent:
\begin{enumerate}
\item The function $g^{(\R)}$ generates a Gabor frame for $L^{2}(\R)$ with respect to the lattice $\alpha\Z\times\beta\Z$.
\item For any prime number $p$ the function $g=g^{(\R)}\otimes \mathds{1}_{\Zp}$ generates a Gabor frame for $L^{2}(\R\times\Qp)$ with respect to the lattice 
\[ \La = \psi_{\alpha}(\Z[1/p])\times \psi_{\beta}(\Z[1/p])= \{ (\alpha q, q, \beta r, r) : q,r \in \Z[1/p] \} \subset \R\times\Qp\times\widehat{\R}\times\widehat{\Qp}.\]
\end{enumerate}
\end{prop}
Thus the results for the noncommutative 2-torus yield also $QC^k$-structures for the Heisenberg modules over noncommutative solenoids. 
Hence we have $QC^k$-structures on the Heisenberg module $\heis{\RR\times\Qp}{\La}$ which does not rely on any kind of derivations on the noncommutative solenoids and indicates the usefulness of modulation spaces in this context.  

\section{Acknowledgement}
The authors would like to thank Adam Rennie for his comments on earlier drafts  of this manuscript. The authors also wish to thank Eirik Skrettingland for help with the examples. 
%
%
%
%\bibliographystyle{abbrv}
%\bibliography{ModSmoothStructures}

\end{document}